\documentclass[12pt]{amsart}
\usepackage{amsmath,amstext,amssymb,amsopn,amsthm}
\usepackage{natbib}
\usepackage{mathtools}
\usepackage{enumerate}

\usepackage{xcolor,graphicx}
\usepackage{mathrsfs,dsfont}

\numberwithin{equation}{section}

\allowdisplaybreaks

\usepackage[utf8]{inputenc}
\usepackage{hyperref}

\makeatletter
\g@addto@macro\th@plain{\thm@headpunct{}}
\makeatother

\numberwithin{equation}{section}

\newtheorem{thm}{Theorem}[section]
\newtheorem*{thm*}{Theorem}
\newtheorem{lemma}[thm]{Lemma}
\newtheorem{Corollary}[thm]{Corollary}
\newtheorem{proposition}[thm]{Proposition}

\theoremstyle{definition}

\newtheorem{defin}[thm]{Definition}

\newtheorem{remark}[thm]{Remark}
\newtheorem{example}[thm]{Example}
\newtheorem{notation}[thm]{Notation}

\def \C {{\mathbb C}}
\def \E {{\mathbb E}}
\def \P {{\mathbb P}}
\def \R {{\mathbb R}}
\def \eps {\varepsilon}
\def \A {{\mathbb A}}

\def \B {{\mathbb B}}
\def \X {{\mathbb X}}

\def \Ab {{\mathbf A}}
\def \Bb {{\mathbf B}}
\def \Xb {{\mathbf X}}
\def \Ub {{\mathbf U}}
\def \Vb {{\mathbf V}}
\def \ub {{\mathbf u}}
\def \vb {{\mathbf v}}
\def \Yb {{\mathbf Y}}

\def \id {{\mathbf{id}}}

\newcommand{\dd}{\mathrm{d}}
\newcommand{\cA}{ {\mathcal A} }
\newcommand{\tcA}{\tilde{\mathcal{A}}}
\newcommand{{\cB}}{ {\mathcal B} }

\newcommand{\cW}{ {\mathcal W} }
\newcommand{\xb}{ {\mathbf x} }
\newcommand{\yb}{ {\mathbf y} }
\newcommand{\ab}{ {\mathbf a} }
\newcommand{\bb}{ {\mathbf b} }

\makeatletter
\@namedef{subjclassname@2020}{%
	\textup{2020} Mathematics Subject Classification}
\makeatother

\title[Spectral distribution of matrix perpetuities]{On the empirical spectral distribution of matrix perpetuities}
\author[B. Ko\l{}odziejek]{Bartosz Ko\l{}odziejek}
\address{Bartosz Ko\l{}odziejek: Faculty of Mathematics and Information Sciences, Warsaw University of Technology, Koszykowa 75, \mbox{00-662} Warsaw, Poland}
\email{bartosz.kolodziejek@pw.edu.pl}

\author[K. Szpojankowski]{Kamil Szpojankowski}
\address{Kamil Szpojankowski: Institute of Mathematics of the Polish Academy of Sciences, ul. \'Sniadeckich 8, 00-656 Warszawa, Poland, and 
	Faculty of Mathematics and Information Science,
	Warsaw University of Technology, Koszykowa 75, 00-662 Warszawa, Poland.} 
\email{kszpojankowski@impan.pl, kamil.szpojankowski@pw.edu.pl  }
\thanks{BK: This research was funded in part by National Science Centre, Poland, 2023/51/B/ST1/01535. \\ KSz: This research was funded in part by National Science Centre, Poland WEAVE-UNISONO grant BOOMER 2022/04/Y/ST1/00008.}


\begin{document}

\begin{abstract}
We study matrix perpetuities, that is, solutions to affine fixed-point equations of the form
\[
\Xb \stackrel{d}{=} \Ab\,\Xb\,\Ab^\top+\Bb,\qquad (\Ab,\Bb)\mbox{ and }\Xb\mbox{ are independent},
\]
with particular emphasis on the empirical spectral distribution of the solution. We first establish existence and uniqueness results by relating the problem to classical vector perpetuities. Under orthogonal invariance, we then prove a compression identity for symmetric multiplicative convolution and show that principal submatrices of a matrix perpetuity are themselves lower-dimensional matrix perpetuities. For positive semidefinite, orthogonally invariant models, we prove a finite-dimensional spectral Kesten theorem: we obtain precise power-law tail asymptotics for the expected empirical spectral distribution, show that its tail is governed by the largest eigenvalue, and relate it explicitly to the tail of any diagonal entry. We also prove that, in the subcritical regime, the expected empirical spectral distribution of matrix perpetuities converges weakly, as the dimension tends to infinity, to the distribution of the corresponding free perpetuity. Our results are illustrated by matrix Beta prime perpetuities, for which explicit limiting spectral distributions are available.

\end{abstract}
\maketitle

\section{Introduction}
\subsection{Matrix perpetuities: definition and existence}

A random vector $\underline{X}$ is called a multivariate perpetuity if its law satisfies the affine fixed-point equation

\begin{align}\label{eqn:classical_perp}
\underline{X} \stackrel{d}{=} \Ab\,\underline{X} + \underline{B},
\end{align}
where  $\Ab$ is an $N\times N$ random matrix, $\underline{B}$ is an $N$-dimensional random vector,
and $\underline{X}$ is assumed to be independent of the pair $(\Ab,\underline{B})$.

The study of perpetuities is a prominent field in probability theory, driven by wide-ranging applications in actuarial science, finance, and queueing theory. Kesten-type affine recursions also arise naturally in random walks in random environments, models of directed polymers in random media, and random-field Ising chains; see \cite[Section~I.A]{MATRIXKESTEN} and the references therein. Once $\underline{X}$ is shown to exist, the focus shifts to understanding how the distribution of $(\Ab, \underline{B})$ determines the distribution of $\underline{X}$, with particular attention given to its asymptotic tail behavior.

The foundational result of Kesten \cite{Kes73} and subsequent developments (see \cite[Section 4.4]{BDM}) state that if there exists a  parameter $\alpha>0$ satisfying
\[
\inf_{n\in\mathbb{N}}\E[ \|\Pi_n\|^\alpha ]^{1/n} = 1,
\]
where $\Pi_n=\Ab_n\Ab_{n-1}\ldots \Ab_1$ and $\Ab_1,\ldots, \Ab_n$ are iid copies of $\Ab$, then under additional technical assumptions on the distribution of $(\Ab,\underline{B})$, there exists a non-zero Radon measure $\mu$ such that the solution to \eqref{eqn:classical_perp} satisfies
\[
t^{\alpha}\,\P\left(t^{-1}\underline{X}\in \cdot\right) \stackrel{v}{\longrightarrow}\mu,\qquad t\to+\infty,
\]
where $\stackrel{v}{\longrightarrow}$ denotes vague convergence of measures.

In the univariate case, the parameter $\alpha>0$ is determined by the equation $\E[|A|^\alpha]=1$. Under additional assumptions, Kesten's theorem guarantees that $\lim_{t\to+\infty}t^\alpha \P(X>t)$ exists and is strictly positive and finite.

In this paper we study matrix fixed point equations of the form
\begin{align}\label{eq:intro}
\Xb\stackrel{d}{=}\Ab \Xb \Ab^\top+\Bb,\qquad (\Ab,\Bb)\mbox{ and }\Xb\mbox{ are independent},
\end{align}
where $\Ab$ and $\Bb$ are square $N\times N$ random matrices with real entries, and $\mathbf{A}^\top$ denotes the transpose of the matrix $\mathbf{A}$.  One may vectorize this equation and apply classical vector perpetuity theory. This point of view is useful for establishing the existence and uniqueness of matrix perpetuities, as we explain in detail in Section \ref{Sec:existence}. The main condition under which a matrix perpetuity exists is $\gamma<0$, where $\gamma$ is the top Lyapunov exponent, defined as
\begin{align*}
\gamma=\inf_{n\geq 1} \frac{1}{n}\E\left[ \log \lambda_{\max}\left(\Ab_1\Ab_2\ldots\Ab_n \Ab_n^\top\ldots\Ab_{2}^\top\Ab_1^\top\right)\right].
\end{align*}

After establishing the existence of matrix perpetuities, we study the empirical spectral distribution (ESD) of $\Xb$. The ESD of a particular matrix Kesten recursion was studied in \cite{MATRIXKESTEN}; apart from this work, to the best of our knowledge, the ESD of matrix perpetuities has not previously been investigated systematically.

\subsection{Motivation and related work}
Before stating our main results, let us briefly indicate one possible motivation
for studying matrix perpetuities. Stochastic gradient descent and related random
iterative algorithms are known to generate heavy-tailed stationary distributions
through Kesten-type affine recursions; see, for instance, \cite{DM24}. In
machine-learning applications, however, one is often interested not only in the
tails of individual coordinates or matrix entries, but also in the tail behavior
of empirical spectral distributions of random matrices; see \cite{ML3}. A
spectral analogue of the classical Kesten mechanism should therefore describe
how multiplicative random matrix effects and additive perturbations shape the
eigenvalue distribution. The matrix perpetuity equation studied in this paper
provides a tractable closed model for this type of phenomenon, without being
intended as an exact model of stochastic gradient descent.

A related beta-type matrix perpetuity arises in the integrable-probability literature. In \cite{ABO}, the authors study discrete analogues of the Matsumoto-Yor and Dufresne identities for multiplicative random walks on the cone of positive definite matrices. In particular, they consider the recursion
\[
\Xb_n=\Ab_n^{1/2}\bigl(I+\Xb_{n-1}\bigr)\Ab_n^{1/2},
\]
where the increments $(\Ab_n)_n$ are iid matrix beta prime random matrices, and identify the stationary distribution through the fixed-point equation
\[
\Xb\stackrel d=\Ab^{1/2}(I+\Xb)\Ab^{1/2},\qquad \Ab\mbox{ and }\Xb\mbox{ are independent}.
\]
This is a special affine matrix perpetuity of the type considered in this paper. This example places matrix perpetuities in the broader context of integrable probability: such recursions arise naturally in matrix analogues of the Matsumoto-Yor and Dufresne identities and are closely related to matrix-valued directed polymer models. Related matrix polymer models with inverse-Wishart disorder were further developed, for instance, in \cite{ABO23,MatrixPolymer}. The beta-prime perpetuity above also serves as a running example throughout the present paper, where it is used to illustrate our general results on spectral distributions, tail behaviour, and high-dimensional limits.

Another matrix-valued Kesten recursion was studied in \cite{MATRIXKESTEN}. The authors consider positive semidefinite matrices satisfying
\[
\mathbf{Z}_{n+1}
=
(\varepsilon I+\mathbf{Z}_n)^{1/2}
\boldsymbol{\xi}_n
(\varepsilon I+\mathbf{Z}_n)^{1/2},
\]
where $(\boldsymbol{\xi}_n)_n$ is an iid sequence of positive semidefinite $N\times N$ random matrices. For their large-dimensional analysis, the authors further assume that the spectrum of $\boldsymbol{\xi}$ remains of order one as $N\to\infty$ and that the distribution is orthogonally invariant; no particular distribution is otherwise imposed in this regime. For the continuum limit at fixed $N$, they specialize to matrix log-normal increments generated by GOE or GUE noise.

Let $\mathbf{Z}$ denote the orthogonally invariant stationary solution to the above recursion. Its stationary fixed-point equation is
\[
\mathbf{Z}
\stackrel{d}{=}
(\varepsilon I+\mathbf{Z})^{1/2}
\boldsymbol{\xi}
(\varepsilon I+\mathbf{Z})^{1/2},\qquad \mathbf{Z}\mbox{ and }\boldsymbol{\xi}\mbox{ are independent.}
\]
The matrix on the right-hand side has the same eigenvalues as
\[
\boldsymbol{\xi}^{1/2}
(\varepsilon I+\mathbf{Z})
\boldsymbol{\xi}^{1/2}.
\]
Since both matrices have orthogonally invariant distributions, equality of their eigenvalue distributions implies equality of their matrix distributions.
Consequently,
\[
\mathbf{Z}
\stackrel{d}{=}
\boldsymbol{\xi}^{1/2}\mathbf{Z}\boldsymbol{\xi}^{1/2}
+\varepsilon\boldsymbol{\xi},\qquad \mathbf{Z}\mbox{ and }\boldsymbol{\xi}\mbox{ are independent.}
\]
Thus the stationary law is an affine matrix perpetuity of the form considered here, with 
\[
(\Ab,\Bb)
=
(\boldsymbol{\xi}^{1/2},\varepsilon\boldsymbol{\xi}).
\]

The authors of \cite{MATRIXKESTEN} study the recursion in the large-dimensional limit using free probability. Rotational invariance leads to asymptotic freeness of $\varepsilon I+\mathbf{Z}$ and $\boldsymbol{\xi}$, so that the limiting spectral evolution can be described through free symmetric multiplicative convolution. Using the multiplicative property of the $S$-transform, they derive a recursion for the limiting spectral distribution and a self-consistent equation for the $S$-transform of a stationary law (see also a related identity for the $S$-transforms in the proof of \cite[Corollary 4.1]{FreePerp}). They solve the corresponding equation in the continuum limit $\varepsilon\to0$, obtaining the inverse-Wishart limiting spectral distribution, and also derive perturbative corrections and moment formulas for the discrete model.

For the discrete recursion at finite dimension, however, \cite{MATRIXKESTEN} leaves open the problem of finding a matrix analogue of the scalar condition determining the heavy-tail exponent. More precisely, in their conclusion the authors ask ``whether the remarkably simple condition $\E[z^\nu]=1$, which determines the heavy tail exponent $\nu$ for the scalar case $N=1$, has any analog in the matrix case.'' We answer this question affirmatively for positive orthogonally invariant matrix perpetuities; see \eqref{eq:simple1} below and the discussion following it. Under our assumptions, this exponent governs the tails of the expected empirical spectral distribution, the largest eigenvalue, and any diagonal entry. Thus, within this class of matrix perpetuities, our results provide a positive solution to the finite-dimensional problem posed in \cite{MATRIXKESTEN}.

\subsection{First result: compression identities and spectral tail asymptotics}

Our finite-dimensional analysis rests on a compression identity for the finite-dimensional symmetric multiplicative convolution, which may also be of independent interest. More precisely, if $\Ab$ and $\Bb$ are independent and $\Bb$ is orthogonally invariant, then
\[
\bigl(\Ab\Bb\Ab^\top\bigr)^{[k]}
\stackrel{d}{=}
\bigl((\Ab\Ab^\top)^{[k]}\bigr)^{1/2}
\Bb^{[k]}
\bigl((\Ab\Ab^\top)^{[k]}\bigr)^{1/2},
\]
where $\mathbf{m}^{[k]}$ denotes the $k\times k$ leading principal submatrix of a matrix $\mathbf{m}$. Equivalently, principal compression is compatible with the finite-dimensional symmetric multiplicative convolution.

Adapting the compression identity, we show that, under orthogonal invariance of the pair $(\Ab,\Bb)$, if the law of $\Xb$ satisfies \eqref{eq:intro}, then the law of the principal $k\times k$ submatrix  $\Xb^{[k]}$ of $\Xb$ satisfies 
\[
 \Xb^{[k]}\stackrel{d}{=}((\Ab\Ab^\top)^{[k]})^{1/2}\Xb^{[k]} ((\Ab\Ab^\top)^{[k]})^{1/2}+\Bb^{[k]},\quad ((\Ab\Ab^\top)^{[k]},\Bb^{[k]})\mbox{ and }\Xb^{[k]}\mbox{ are independent.}
\]
In particular, for $k=1$, we obtain a scalar perpetuity equation
\begin{align*}
 \Xb_{11}\stackrel{d}{=} (\Ab\Ab^\top)_{11}\Xb_{11}+\Bb_{11},\qquad((\Ab\Ab^\top)_{11}, \Bb_{11})\mbox{ and }\Xb_{11}\mbox{ are independent}.
 \end{align*}
 This reduction allows us to apply scalar Kesten theory to diagonal entries of a matrix perpetuity. For positive semidefinite, orthogonally invariant models, we obtain a finite-dimensional spectral Kesten theorem. Under the condition \begin{align}\label{eq:simple1}
\E\left[((\Ab\Ab^\top)_{11})^\eta\right]=1
\end{align}
and suitable additional assumptions, the expected empirical spectral distribution has a regularly varying tail with exponent $\eta$. We show that this tail is asymptotically governed by the largest eigenvalue and relate it explicitly to the tail of any diagonal entry. In particular, the extreme event contributing to the spectral tail is asymptotically generated by a single eigenvalue.

\subsection{Second result: high-dimensional convergence to free perpetuities}
In \cite{FreePerp} we studied the free perpetuities, that is, solutions to the following fixed point equations 
\[
        \X \stackrel{d}{=} \A\,\X\,\A^* + \B,
\]
where $\X$ is $*$-free from the pair $(\A,\B)$, where $\B=\B^*$. As we have shown in \cite{FreePerp}, the free perpetuity equation is equivalent to $\X \stackrel{d}{=} |\A|\X\,|\A| + \B$, and we considered without loss of generality an equation of the form $\X \stackrel{d}{=} \A^{1/2}\X\,\A^{1/2} + \B$, assuming that $\A\geq 0$. We found sufficient conditions under which free perpetuity exists. In particular, if $\tau(\A)<1$ and $\B\geq 0$, then the free perpetuity exists. It also exists in the critical case $\tau(\A)=1$, provided that $\A$ is non-Dirac. The free analogue of Kesten’s criticality condition is $\tau(\A)=1$. We showed, for example, that under this condition, if $\mathrm{Var}(\A)<\infty$ then $ \mu_{\X}\big((t,+\infty)\big)\sim c \,t^{-1/2}$, with an explicit constant $c>0$.

From the point of view of matrix perpetuities, it is natural to ask how the ESD of matrix perpetuities connects to the distribution of the free perpetuity. More precisely, suppose that the sequence of random matrices $(\Ab_N,\Bb_N)$ converges in noncommutative distribution (for a precise definition, see Section~\ref{Sec:preliminaries}) to $(\A,\B)$. In this framework, we would like to understand whether, under suitable conditions, the expected ESD of $\Xb_N$ converges weakly to the distribution of $\X$. Asymptotic freeness emerges naturally when independent random matrices satisfy additional invariance of the distribution, therefore we assume throughout the paper that 
\begin{align}\label{eqn:intro_invariance}
    (\ub \Ab \ub^\top, \ub \Bb \ub^\top)\stackrel{d}{=}(\Ab,\Bb)\quad\mbox{or}\quad (\ub \Ab\Ab^\top \ub^\top, \ub \Bb \ub^\top)\stackrel{d}{=} (\Ab\Ab^\top, \Bb) 
\end{align} for any orthogonal matrix $\ub$. We show that in the subcritical regime $\tau(\A)<1$ indeed we have weak convergence of the expected ESD of $\Xb_N$ to the distribution of $\X$. In view of the connection between beta-type matrix perpetuities and matrix polymer models discussed above, this convergence result also suggests a possible route toward defining and studying free analogues of matrix polymers.

Our large-dimensional result complements the free-probability analysis of \cite{MATRIXKESTEN}.  For the special choice
\[
(\Ab,\Bb)
=
(\boldsymbol{\xi}^{1/2},\varepsilon\boldsymbol{\xi}),
\]
the limiting free fixed-point equation coincides with the stationary equation underlying the analysis of \cite{MATRIXKESTEN}.

It is important to note that tail asymptotics are not continuous under weak convergence of probability measures. Indeed, any probability measure $\mu$, regardless of its tail behaviour, can be approximated weakly by a sequence of compactly supported probability measures $(\mu_N)_{N\geq1}$. Consequently, although $\mu{\Xb_N}$ converges weakly to the law $\mu_{\X}$ of the free perpetuity, this convergence alone does not imply convergence of their tail asymptotics. 
 
\subsection{Orthogonal invariance and similarity assumptions} It is useful to distinguish our orthogonal invariance assumption \eqref{eqn:intro_invariance} from the similarity assumptions appearing in the multidimensional perpetuity literature. The action of an orthogonal matrix $\ub \in O(N)$ on \(\R^{N\times N}\) by conjugation is represented, after vectorization, by
\[
\mathrm{vec}(\ub\,\mathbf m\,\ub^\top)=\rho(\ub)\mathrm{vec}(\mathbf m),
\qquad \rho(\ub)=\ub\otimes\ub,
\]
where $\otimes$ is the Kronecker product. Thus, at the vectorized level, the corresponding orthogonal invariance may be written as
\[
(\rho(\ub)\tilde{\Ab}\rho(\ub)^\top,\rho(\ub)\mathrm{vec}(\Bb))\stackrel{d}{=}(\tilde{\Ab},\mathrm{vec}(\Bb))
\]
for any orthogonal matrix $\ub\in O(N)$, where  $\tilde{\Ab}:= \Ab\otimes \Ab$. 
Equivalently, this is invariance under the subgroup 
\[
\rho(O(N))=\{\ub\otimes \ub:\ub\in O(N)\}\subset O(N^2).
\]
This is different from the framework of \cite{Sim09,Sim10}, where the linear coefficients of a vector recursion $\R^d \ni x\mapsto \mathbf{M} \,x+Q$ are assumed to be similarities, namely
\[
\mathbf{M}=a\,\mathbf{U},\quad\mbox{where}\quad a>0,\quad \mathbf{U}\in O(d).
\]
In order to compare this with the full vectorization of our recursion, one should take \(d=N^2\),
so that the orthogonal factor \(\mathbf{U}\) acts on the whole vectorized space \(\R^{N^2}\). By contrast, in our matrix-valued recursion the relevant orthogonal action is only the conjugation action inherited from the underlying \(N\)-dimensional matrix structure. 
Although $\rho(O(N))$ is a proper subgroup of $O(N^2)$ for $N\geq 2$, this group inclusion alone does not make our assumption weaker. The two frameworks impose assumptions of a different nature: the similarity condition restricts every realization of the vectorized coefficient to lie in $\mathbb{R}_{+}O(N^2)$, whereas our condition imposes distributional invariance under the conjugation representation $\rho(O(N))$. Moreover, the matrices $\Ab\otimes\Ab$ need not be similarities. Thus, neither framework is a direct special case of the other.

\subsection{Organization of the paper}
The paper is organized as follows. In Section~\ref{Sec:preliminaries} we introduce the notation and collect the necessary background on empirical spectral distributions, noncommutative probability, and free perpetuities. In Section~\ref{Sec:existence} we establish existence and uniqueness of matrix perpetuities, discuss equivalent multiplication schemes under orthogonal invariance, and relate matrix perpetuities to their principal minors. Section~\ref{sec:tails} is devoted to tail asymptotics of the expected empirical spectral distribution of matrix perpetuities. In Section~\ref{sec:weak} we study the convergence of expected empirical spectral distributions of matrix perpetuities to the law of the corresponding free perpetuity. Appendix~\ref{app:A} contains a detailed analysis of the matrix Beta prime example and its limiting spectral distribution. Appendix~\ref{app:B} records a marginal property of the matrix Beta prime distribution for principal submatrices. Finally, Appendix~\ref{sec:C} provides the measurable-selection argument used to construct a measurable polar decomposition.

\section{Preliminaries}\label{Sec:preliminaries}
\begin{notation}
Throughout the paper, bold lowercase letters, such as ${\bf x}$, denote deterministic matrices, whereas bold uppercase letters, such as ${\bf X}$, denote random matrices. The $N \times N$ identity matrix is denoted by $\id_N$.

We use the following matrix spaces:
\begin{align*}
\mathrm{Mat}(N):= \mathbb{R}^{N \times N}, \qquad 
\mathrm{Sym}(N) := \{ {\bf x} \in \mathrm{Mat}(N) \colon {\bf x} = {\bf x}^\top \}.
\end{align*}

The cones of positive semidefinite and positive definite matrices are denoted by
\begin{align*}
\mathrm{Sym}_{\geq 0}(N) 
&:= \{ {\bf x} \in \mathrm{Sym}(N) \colon {\bf x} \geq 0 \}, \\
\mathrm{Sym}_{> 0}(N) 
&:= \{ {\bf x} \in \mathrm{Sym}(N) \colon {\bf x} > 0 \}.
\end{align*}
Here, ${\bf x} \geq 0$ means that ${\bf x}$ is positive semidefinite, while ${\bf x} > 0$ means that ${\bf x}$ is positive definite.

The orthogonal group is denoted by
\[
O(N) := \{ {\bf u} \in \mathrm{Mat}(N) \colon {\bf u}{\bf u}^\top = \id_N \}.
\]
For ${\bf x} \in \mathrm{Mat}(N)$, we denote by $\mathrm{Tr}({\bf x})$ its trace and by $\mathrm{tr}({\bf x})$ its normalized trace, that is,
\[
\mathrm{tr}({\bf x}) = \frac{1}{N}\mathrm{Tr}({\bf x}).
\]
\end{notation}

\begin{defin}
We say that the distribution of an $N\times N$ real random matrix $\Yb$ is invariant under the action of the orthogonal group $O(N)$ if, for any orthogonal matrix ${\bf u}$, we have
    \[
    {\bf u}\Yb {\bf u}^\top\stackrel{d}{=}\Yb.
    \]    
More generally we say that a pair of random matrices $(\Ab,\Bb)$ is orthogonally invariant if for every orthogonal matrix $\ub$, we have
\[(\Ab,\Bb)\stackrel{d}{=}\left(\ub\Ab\ub^\top,\ub\Bb\ub^\top\right).\]
\end{defin}

\begin{defin}\label{def:spect}\ 
    For any $N\times N$ real  random matrix $\Yb$ with eigenvalues $(\lambda_k(\Yb))_{k=1}^N$, we define the empirical spectral distribution and the expected empirical spectral distribution by $\hat{\mu}_{\Yb}$ and $ \mu_{\Yb}$ respectively, by 
    \[
    \hat{\mu}_\Yb = \frac{1}{N}\sum_{k=1}^N \delta_{\lambda_k(\Yb)},\qquad \mu_{\Yb}=\E[\hat{\mu}_\Yb], 
    \]
    where $\lambda_k(\Yb)$ for $k=1,\ldots,N$ are the eigenvalues of $\Yb$ listed with multiplicities.

\end{defin}

\subsection{Noncommutative probability}

A noncommutative probability space is a pair $(\mathcal{A},\tau)$ where $\mathcal{A}$ is a $*$-algebra, and $\tau$ is a faithful, normal, tracial state. 

We will use two noncommutative probability spaces. In the $W^*$-setting,  $\mathcal{A}$ is a finite von Neumann algebra, $\tau$ is faithful, normal tracial state, and $\tcA$ denotes the algebra of affiliated operators. We also use the random-matrix space, of random matrices whose entries have all moments finite, and $\tau=\E\otimes \mathrm{tr}$.

For fixed, selfadjoint elements $\A_1,\ldots,\A_n\in \mathcal{A}$ their joint distribution is a functional $\mu_{\A_1,\ldots,\A_n}:\C \langle X_1,\ldots,X_N \rangle\to \C$, which for any  noncommutative polynomial $P\in \C \langle X_1,\ldots,X_N \rangle$ is defined via
\[\mu_{\A_1,\ldots,\A_n}(P)=\tau(P(\A_1,\ldots,\A_n)).\]

Let $\left(\Ab_1^{(N)},\ldots,\Ab_n^{(N)}\right)_{N\geq 1}$ be a sequence of $n$-tuples of random matrices, where $N$-th term in this sequence consists of $N\times N$ matrices. We say that $\left(\Ab_1^{(N)},\ldots,\Ab_n^{(N)}\right)_{N\geq 1}$ converges in noncommutative distribution to a tuple $(\A_1,\ldots,\A_n)$ in $(\mathcal{A},\tau)$ if
for every noncommutative polynomial $P\in\C\langle X_1,\ldots,X_{n}\rangle$ we have
\[
\lim_{N\to\infty}\E[\mathrm{tr}(P(\Ab^{(N)}_1,\ldots,\Ab_n^{(N)}))]= \tau(P(\A_1,\ldots,\A_n)).
\]

\subsection{Free perpetuities}

A free perpetuity is a self-adjoint operator \(\X\) affiliated with
\(\mathcal{A}\) satisfying the affine fixed-point equation
\begin{align}\label{eq:affine}
\X \stackrel{d}{=} \A^{1/2}\, \X \,\A^{1/2} + \B,
\qquad \X\mbox{ and }
(\A,\B)\ \text{are $*$-free},
\end{align}
where \(\A,\B\in\mathcal{A}\), $\A=\A^\ast\geq 0$ and \(\B=\B^\ast\neq 0\). For simplicity, we work only with bounded coefficients \(\A,\B\in\mathcal{A}\), while \cite{FreePerp} allows \(\A\) and \(\B\) to be unbounded operators affiliated with \(\mathcal{A}\).

\begin{thm}\label{thm:uniqe_free}
    Assume $\B\geq0$ and $\B\in\cA$. If 
    \begin{itemize}
        \item $\tau(\A)<1$, or 
        \item $\tau(\A)=1$ and $\A$ is non-Dirac,
    \end{itemize} 
    then there exists a unique solution to  \eqref{eq:affine}.
\end{thm}
The case $\tau(\A)<1$ follows from \cite[Theorem 4.6(ii)]{FreePerp}, while the critical case $\tau(\A)=1$ follows from \cite[Theorem 3.15]{FreePerp}. Uniqueness is proved in \cite[Theorem 4.2]{FreePerp}.

\section{Matrix fixed-point affine equation}\label{Sec:existence}
Let $\Ab$ and $\Bb$ be square $N\times N$ random matrices with real entries. We consider an affine stochastic equation of the form
\begin{align}\label{eq:matperp}
\Xb\stackrel{d}{=}\Ab \Xb \Ab^\top+\Bb,\qquad (\Ab,\Bb)\mbox{ and }\Xb\mbox{ are independent},
\end{align}
where $\mathbf{A}^\top$ denotes the transpose of a matrix $\mathbf{A}$. 

To analyze this equation, we reformulate the problem in the usual vector setting, see \cite{BDM}. By vectorization of an $N\times N$ matrix ${\bf x}=[x_1,\ldots,x_N]\in \mathrm{Mat}(N)$ we mean the vector
\[\mathrm{vec}(\xb)=\begin{bmatrix}
    x_1\\\vdots\\x_N
\end{bmatrix} \in\R^{N^2}.\]
Matrix equation \eqref{eq:matperp} is equivalent to the following
\begin{align}\label{eq:vecperp}
\mathrm{vec}(\Xb) \stackrel{d}{=} \tilde{\Ab}\mathrm{vec}(\Xb)+\mathrm{vec}(\Bb),\qquad (\tilde{\Ab},\mathrm{vec}(\Bb))\mbox{ and }\mathrm{vec}(\Xb)\mbox{ are independent},
\end{align}
where $\tilde{\Ab}=\Ab\otimes\Ab$ with $\otimes$ denoting the Kronecker product. Using this approach, we can establish conditions on the distributions of $\Ab$ and $\Bb$ under which a unique solution to \eqref{eq:matperp} exists. These conditions do not impose structural assumptions on the joint distribution of the pair $(\Ab,\Bb)$.

\subsection{Existence and uniqueness of matrix perpetuities}
For $a>0$, define $\log^+(a)=\max\{\log(a),0\}$. Let $\lambda_{\max}(\ab)$ denote the largest eigenvalue of a symmetric matrix $\ab$. 
\begin{thm}\label{thm:matrixperp}
Let $(\Ab_n,\Bb_n)_{n\geq 1}$ be iid copies of $(\Ab,\Bb)$. If 
\begin{align}\label{eq:log_moments}
\E\left[ \log^+\left(\lambda_{\max}(\Ab\Ab^\top)\right)\right]<+\infty\quad\mbox{and}\quad \E\left[ \log^+\left(\lambda_{\max}(\Bb\Bb^\top)\right)\right]<+\infty 
\intertext{and}
\gamma=\inf_{n\geq 1} \frac{1}{n}\E\left[ \log( \lambda_{\max}\left(\Ab_1\Ab_2\ldots\Ab_n \Ab_n^\top\ldots\Ab_{2}^\top\Ab_1^\top\right))\right] <0, \label{eq:topL}
\end{align}
then there exists a unique solution to \eqref{eq:matperp}. 

Moreover, this solution has a series representation
\begin{align}\label{eq:series1}
\Xb\stackrel{d}{=} \sum_{n=1}^\infty \Ab_1\ldots \Ab_{n-1}\Bb_n\Ab_{n-1}^\top\ldots\Ab_1^\top,
\end{align}
where the series on the right-hand side above converges almost surely, with $(\Ab_n,\Bb_n)_{n\ge1}$  iid copies of $(\Ab,\Bb)$. 

Finally, letting
\[
\Xb_n=\Ab_n \Xb_{n-1}\Ab_n^\top+\Bb_n,\qquad n\ge1,
\]
where $\Xb_0$ is independent of $(\Ab_n,\Bb_n)_{n\ge1}$, we have 
\[
\Xb_n \stackrel{d}{\longrightarrow} \Xb.
\]
\end{thm}
\begin{proof}
Let $d=N^2$.  We define a norm $|\cdot|$ on $\R^d$ by  $|\mathrm{vec}(\mathbf{x})|=\sqrt{\mathrm{Tr}(\mathbf{x}\mathbf{x}^\top)}$ for any $N\times N$ matrix $\mathbf{x}$. The operator norm of $\textbf{a}\otimes\textbf{a}$ induced by $|\cdot|$, where $\mathbf{a}$ is an $N\times N$ matrix, is given by  
\[
\|\textbf{a}\otimes\textbf{a}\| = \sup_{\mathbf{x}\neq 0} \frac{|(\textbf{a}\otimes\textbf{a})\mathrm{vec}(\mathbf{x})|}{|\mathrm{vec}(\mathbf{x})|} = \sqrt{
\sup_{\mathbf{x}\neq 0} \frac{\mathrm{Tr}(\mathbf{x}\,\mathbf{a}^\top \mathbf{a}\,\mathbf{x}^\top\,\mathbf{a}^\top\mathbf{a})}{\mathrm{Tr}(\mathbf{x}\mathbf{x}^\top)}
}=
\lambda_{\max}(\mathbf{a}^\top \mathbf{a}),
\]
this can be derived as follows: diagonalize $\mathbf{a}^\top \mathbf{a}=\mathbf{u} \mathbf{d} \mathbf{u}^\top$ and substitute $\mathbf{x}=\mathbf{u}^\top \mathbf{y} \mathbf{u}$, where $\mathbf{u}$ is an orthogonal matrix and $\mathbf{d}=\mathrm{diag}(d_1,\ldots,d_N)$ is a diagonal matrix with $d_1\geq\ldots\geq d_N\geq 0$. Then $\mathrm{Tr}(\mathbf{ydy}^T\mathbf{d})=\sum_{i,j} d_i d_j y_{ij}^2$, and 
$\mathrm{Tr}(\mathbf{yy}^T)=\sum_{i,j} y_{ij}^2$, and the quotient is maximized by $d_1^2$.

By \cite[Theorem 4.1.4]{BDM}, a unique solution to \eqref{eq:vecperp} exists if
 \[
\E[\log^+(\|\tilde \Ab\|)]<+\infty\quad\mbox{and}\quad \E[\log^+(|\mathrm{vec}(\Bb)|)]<+\infty.
\]
and
\[
\gamma:=\inf_{n\geq 1} \frac{1}{n}\E\left[ \log(\|\tilde\Ab_1\cdots \tilde\Ab_n\|)\right] <0.
\]
Note that the above result does not depend on the choice of the norm $|\cdot|$

Using the explicit expressions for our choice of $|\cdot|$ and $\|\cdot\|$, and the fact that
\[
\E\left[ \log^+\left(\mathrm{Tr}(\Bb\Bb^\top)\right)\right]\leq \E\left[ \log^+\left(N\lambda_{\max}(\Bb\Bb^\top)\right)\right]\leq \E\left[ \log^+\left(\lambda_{\max}(\Bb\Bb^\top)\right)\right] + \log(N),
\]
we obtain the first part of the assertion. Furthermore, by the second part of \cite[Theorem 4.1.4]{BDM}, the solution to \eqref{eq:vecperp} has a series representation
\begin{align}\label{eq:series2}
\mathrm{vec}(\Xb)\stackrel{d}{=} \sum_{n=1}^\infty \tilde\Ab_1\ldots \tilde\Ab_{n-1}\mathrm{vec}(\Bb_n),
\end{align}
where $(\tilde{\Ab}_n,\mathrm{vec}(\Bb_n))_n$ are iid copies of $(\tilde{\Ab},\mathrm{vec}(\Bb))$, and the series on the right-hand side above converges almost surely.
Clearly, both series representations \eqref{eq:series1} and \eqref{eq:series2} are equivalent.

Applying the convergence part of \cite[Theorem 4.1.4]{BDM} to the vectorized recursion 
\[
\mathrm{vec}(\Xb_n) =\tilde{\Ab}_n\mathrm{vec}(\Xb_{n-1})+\mathrm{vec}(\Bb_n),
\]
we obtain $\mathrm{vec}(\Xb_n)\stackrel{d}{\longrightarrow} \mathrm{vec}(\Xb)$; hence $\Xb_n\stackrel{d}{\longrightarrow} \Xb$.
\end{proof}

If $\Bb=\id_N-\Ab\Ab^\top$ a.s., then $\Xb\stackrel{d}{=}\id_N$ is a solution to \eqref{eq:matperp}. To prevent such a degeneracy of the model, the notion of so-called irreducibility is introduced, see \cite{BP92}. 
\begin{defin}\label{def:irred}
    We say that an affine subspace $H$ of $\mathrm{Mat}(N)$ is $(\Ab,\Bb)$-invariant if 
\[
\{\Ab \mathbf{x}\Ab^\top+\Bb\colon \mathbf{x}\in H\}\subset H\quad\mbox{a.s.}
\]
We say that the model \eqref{eq:matperp} is irreducible if the entire space $\mathrm{Mat}(N)$ is the only $(\Ab,\Bb)$-invariant subspace.
\end{defin}
 If the model is non-irreducible, then $(\Ab,\Bb)$ acts (as described in the definition above) on a proper subspace $V$ of $\mathrm{Mat}(N\times N)$, and there exists a (non-random) element $\mathbf{x}_0\in V^\bot$ for which $\Ab \mathbf{x}_0\Ab^\top+\Bb\mid_{V}=\mathbf{x}_0$ almost surely, \cite[Lemma 4.1.1]{BDM}.  Under irreducibility, a partial converse to Theorem \ref{thm:matrixperp} can be obtained; see \cite[Theorem 4.1.6]{BDM}.
\begin{lemma}
    Assume that condition \eqref{eq:log_moments} holds and that the model \eqref{eq:matperp} is irreducible. If there exists a solution to \eqref{eq:matperp}, then $\gamma$, as defined in \eqref{eq:topL}, is negative.
\end{lemma}
Although Theorem \ref{thm:matrixperp} provides a sufficient condition for existence and yields a series representation of a perpetuity, it generally does not lead to an explicit description of its distribution. Explicit distributions are rare. We present two known examples where existence follows from the theorem above and for which the distribution can be computed explicitly. We leave the verification of the assumptions to the reader.

\begin{example}\label{ex:1}
Define $d_N=N^{-1}\dim\mathrm{Sym}(N)=(N+1)/2$.
The matrix beta-prime distribution with parameters $\alpha,\beta>d_N-1$ on the cone of positive definite matrices $\mathrm{Sym}_{\geq0}(N)$, is defined by its density
\begin{align}\label{eq:B'N}
\mathcal{B}'_{\alpha,\beta}(N)(\dd\mathbf{x}) = \frac{1}{\mathrm{B}_N(\alpha,\beta)} \frac{\det(\mathbf{x})^{\alpha-d_N}}{\det(\id_N+\mathbf{x})^{\alpha+\beta}}I_{\mathrm{Sym}_{\geq0}(N)}(\mathbf{x})\dd\mathbf{x},
\end{align}
where $\dd\mathbf{x}$ is the Lebesgue measure on $\mathrm{Sym}(N)$ endowed with the inner product $\langle \xb,\yb\rangle=\mathrm{Tr}(\xb \yb)$, and $\mathrm{B}_N(\alpha,\beta)$ is the normalizing constant. 
In \cite[Section 5]{ABO} the authors showed that
    if $\Ab\sim \mathcal{B}'_{\alpha,\alpha+\beta}(N)$ with $\alpha,\beta>d_N-1$ then, $\Xb \sim \mathcal{B}'_{\alpha,\beta}(N)$ is the unique solution to the following stochastic fixed-point equation
\begin{align}\label{eq:BetaPerp}
\Xb \stackrel{d}{=} \Ab^{1/2} \Xb\Ab^{1/2}+\Ab,\qquad \Ab\mbox{ and }\Xb \mbox{ are independent}.
\end{align}
\end{example}

\begin{example}
On $\mathrm{Sym}(N)$, we define the Wishart distribution with parameters $p>d_N-1$ and $\ab\in\mathrm{Sym}_{>0}(N)$ by 
    \begin{align}\label{eq:Wishart}
    \cW_{p,\ab}(N)(\dd \xb) = \frac{\det(\ab)^p}{\Gamma_N(p)}\det(\xb)^{p-d_N} e^{-\mathrm{Tr}(\ab\xb)}I_{\mathrm{Sym}_{>0}(N)}(\xb)\dd\xb
    \end{align}
and the matrix generalized inverse Gaussian distribution with parameters $p\in\R$, $\ab,\bb\in\mathrm{Sym}_{>0}(N)$ by
     \begin{align}\label{eq:GIG}
    \mathcal{G}_{p,\ab,\bb}(N)(\dd \xb) = \frac{1}{K_{N,p}(\ab,\bb)} \det(\xb)^{p-d_N} e^{-\mathrm{Tr}(\ab\xb)-\mathrm{Tr}(\bb\xb^{-1})}I_{\mathrm{Sym}_{>0}(N)}(\xb)\dd\xb,
    \end{align}
where $\Gamma_{N}(p)$ and $K_{N,p}(\ab,\bb)$ are normalizing constants. 

According to \cite{LW00} (see also \cite{KMY} for extension to other symmetric cones), 
if $\Xb\sim \mathcal{G}_{-p,\ab,\bb}(N)$ and $\Yb\sim \cW_{p,\ab}(N)$ are independent, then 
\[
\mathbf{V}:=\Xb^{-1}-(\Xb+\Yb)^{-1}\sim \cW_{p,\bb}. 
\]
Thus, if $\ab=\bb$, then $\Yb\stackrel{d}{=}\mathbf{V}$, which by Hua's identity implies that
\[
\Yb^{-1}\stackrel{d}{=}\mathbf{V}^{-1} = \Xb \Yb^{-1}\Xb+\Xb, 
\]
providing another example of an equation of the form \eqref{eq:matperp} with an explicit solution. 

A generalization of this example to Jordan algebras was recently considered in \cite[Theorem 2.4]{MY2024}.
\end{example}

\section{Orthogonal invariance}
\subsection{Multiplication schemes}
Next we observe that in the case of orthogonally invariant models, one can equivalently consider different multiplication schemes than the one in \eqref{eq:matperp}.
It is customary to denote $(\Ab \Ab^\top)^{1/2}$ by $|\Ab^\top|$. In a slight abuse of notation, we will instead write this quantity as $|\Ab|$.

\begin{defin}
   By $m\colon \mathrm{Mat}(N)\to \mathrm{Mat}(N)$ we denote a measurable function satisfying $m(\ab)m(\ab)^\top = \ab \ab^\top$. 
\end{defin}

Examples of $m$ include $m_1(\ab) = (\ab \ab^\top)^{1/2}$ and $m_2(\ab) = T_{\ab \ab^\top}$, where $T_\bb$ is lower triangular matrix from the Cholesky decomposition of $\bb=T_{\bb}T_{\bb}^\top$. 

\begin{lemma}\label{lem:m_polar}
Let $\Ab$ be a random real matrix. Then there exists a $\sigma(\Ab)$-measurable orthogonal matrix $\Ub$ such that
\[
\Ab = m(\Ab)\Ub.
\]
\end{lemma}

\begin{proof}
By Lemma \ref{lem:measurable}, there exist orthogonal matrices $\Ub_1,\Ub_2$, measurable with respect to $\sigma(\Ab)$ and $\sigma(m(\Ab))$ respectively, such that
\[
\Ab=|\Ab|\,\Ub_1,
\qquad
m(\Ab)=|m(\Ab)|\,\Ub_2.
\]
Since $m(\Ab)m(\Ab)^\top=\Ab\Ab^\top$, we have $|m(\Ab)|=|\Ab|$. Hence
\[
\Ab=|\Ab|\,\Ub_1 = |m(\Ab)|\,\Ub_1 = m(\Ab)\Ub_2^\top \Ub_1.
\]
Because $m$ is measurable, $\sigma(m(\Ab))\subset \sigma(\Ab)$, so $\Ub_2^\top \Ub_1$ is $\sigma(\Ab)$-measurable. Therefore, setting
\[
\Ub:=\Ub_2^\top \Ub_1,
\]
we obtain the claim.
\end{proof}

\subsection{Principal minors and multiplicative convolution}

We begin with a simple observation which, when specialized to the Cholesky factorization, has useful consequences for the distributions of principal minors.

\begin{proposition}\label{prop:390}
Let $\Ab$ and $\Bb$ be independent. Assume that $\Bb$ is orthogonally invariant. Then, 
\[
\Ab \Bb \Ab^\top  \stackrel{d}{=} m(\Ab)\Bb m(\Ab)^\top. 
\]
\end{proposition}
\begin{proof}

By Lemma \ref{lem:m_polar}, there exists a $\sigma(\Ab)$-measurable orthogonal matrix $\Ub_\Ab$ such that $\Ab = m(\Ab) \Ub_{\Ab}$.
Consider 
\[
\Ab \Bb \Ab^\top = m(\Ab) \Ub_{\Ab} \Bb \Ub_{\Ab}^\top m(\Ab)^\top. 
\]
Note that we have $(\Ab, \Ub_\Ab \Bb \Ub_{\Ab}^\top)\stackrel{d}{=}(\Ab,\Bb)$. Indeed, 
\[
(\Ub_\Ab \Bb \Ub_{\Ab}^\top\mid \Ab=\ab )= \Ub_\ab \Bb \Ub_{\ab}^\top\stackrel{d}{=} \Bb.
\]
Thus, we obtain 
\[
\Ab \Bb \Ab^\top  \stackrel{d}{=} m(\Ab)\Bb m(\Ab)^\top. 
\]
\end{proof}

\begin{notation}
For any $N\times N$ matrix $\xb$ and $k=1,\ldots,N$, we denote by $\xb^{[k]}$ the $k\times k$ principal submatrix of $\xb$, i.e., $\xb^{[k]} = (\xb_{ij})_{i,j\leq k}$.
\end{notation}

\begin{Corollary}[Compression identity]
\label{prop:39}
Let $\Ab$ and $\Bb$ be independent. Assume that $\Bb$ is orthogonally invariant. Then, for $k=1,\ldots,N$,
\[
(\Ab \Bb \Ab^\top)^{[k]}\stackrel{d}{=}   ((\Ab \Ab^\top)^{[k]})^{1/2} \Bb^{[k]} ((\Ab \Ab^\top)^{[k]})^{1/2}.
\]
\end{Corollary}
\begin{proof}
By Proposition \ref{prop:390}, we have 
\[
\Ab \Bb \Ab^\top  \stackrel{d}{=} m(\Ab)\Bb m(\Ab)^\top. 
\]
Take now $m(\Ab)=T_{\Ab \Ab^\top}$, where $T_{\ab}$ is the lower triangular matrix from the Cholesky decomposition of $\ab=T_{\ab}T_{\ab}^\top$. 

Let $\xb^{[k]}:=(\xb_{ij})_{i,j\leq k}$ be the $k\times k$ leading principal submatrix of $\xb$. It is well known that  
\begin{align}\label{eq:chol_prop}
({\bf t} \xb {\bf t}^\top)^{[k]} = {\bf t}^{[k]} \xb^{[k]} ({\bf t}^{[k]})^\top
\end{align}
for any lower triangular matrix $\bf t$ and any matrix $\xb$. Thus, we obtain 
\[
(\Ab \Bb \Ab^\top)^{[k]}\stackrel{d}{=} T_{\Ab \Ab^\top}^{[k]} \Bb^{[k]} (T_{\Ab \Ab^\top}^{[k]})^\top
\]
By \eqref{eq:chol_prop}
\[
T_{\Ab \Ab^\top}^{[k]}=T_{(\Ab \Ab^\top)^{[k]}}.
\]
Since $\Bb^{[k]}$ is also orthogonally invariant, we have 
\[
 T_{\Ab \Ab^\top}^{[k]} \Bb^{[k]} (T_{\Ab \Ab^\top}^{[k]})^\top \stackrel{d}{=}  ((\Ab \Ab^\top)^{[k]})^{1/2} \Bb^{[k]} ((\Ab \Ab^\top)^{[k]})^{1/2}
\]
by Proposition \ref{prop:390}.
\end{proof}

\begin{remark}
Let
\[
C_k:\mathrm{Sym}(N)\to \mathrm{Sym}(k),
\qquad
C_k(\ab):=\ab^{[k]},
\]
and write \((C_k)_\#\mu\) for the
push-forward of a probability measure \(\mu\) under \(C_k\).

For probability measures \(\mu,\nu\) on \(\mathrm{Sym}(N)\), with
\(\mu\) supported on positive semidefinite matrices, define
\(\mu\circledast\nu\) as the law of
\[
\Ab^{1/2}\Bb\Ab^{1/2},
\]
where \(\Ab\sim\mu\), \(\Bb\sim\nu\), and \(\Ab,\Bb\) are independent.

Suppose now that \(\nu\) is orthogonally invariant, that is,
\[
\mathrm{Law}(\Bb)=\mathrm{Law}(\ub\Bb\ub^\top),
\qquad \ub\in O(N).
\]
Then Corollary~\ref{prop:39} implies that, for every \(k=1,\ldots,N\),
\[
(C_k)_\#(\mu\circledast\nu)
=
(C_k)_\#\mu \circledast (C_k)_\#\nu,
\]
where the convolution on the right-hand side is understood in dimension
\(k\).

Thus the principal compression map \(C_k\) is compatible with the
symmetric multiplicative convolution \(\circledast\). In particular, on
the class of orthogonally invariant measures, \(C_k\) acts as a
homomorphism with respect to \(\circledast\).
\end{remark}

\begin{remark} The preceding finite-dimensional compatibility has an asymptotic analogue. 
It is well known (see \cite{NicaSpeicherMult,ShlyakhtenkoTao}) that if we have a sequence of orthogonally invariant matrices $\Ab_N\in\mathrm{Mat}(N)$ such that the empirical spectral distribution (ESD) $\mu_{\Ab_N}$ converges weakly to $\mu$, then for any fixed $t>1$ the ESD of $t\,\Ab_N^{[k]}$, with $k=\lfloor N/t\rfloor$, converges to the fractional free convolution power $\mu^{\boxplus t}$.

Moreover, assume that we have a sequence of  orthogonally invariant random matrices $\Bb_N$ independent of $\Ab_N$, whose ESD converges to $\nu$. From the previous Corollary we obtain
\[
(\Ab_N^{1/2}\Bb_N\Ab_N^{1/2})^{[k]}
\stackrel{d}{=}
(\Ab_N^{[k]})^{1/2}\Bb_N^{[k]}(\Ab_N^{[k]})^{1/2}.
\]
As $N\to\infty$, the ESD of 
\[
t\,(\Ab_N^{1/2}\Bb_N\Ab_N^{1/2})^{[k]}
\]
converges to $(\mu\boxtimes\nu)^{\boxplus t}$. On the other hand, multiplying the right-hand side of the above equality by $t^2$, we obtain
\[
(t\Ab_N^{[k]})^{1/2}\,t\Bb_N^{[k]}\,(t\Ab_N^{[k]})^{1/2},
\]
whose ESD converges to $\mu^{\boxplus t}\boxtimes\nu^{\boxplus t}$. Therefore,
\[
\mu^{\boxplus t}\boxtimes\nu^{\boxplus t}
=
(\mu\boxtimes\nu)^{\boxplus t}\circ D_{1/t},
\]
where $\mu\circ D_r(A)=\mu(rA)$ denotes dilation of a measure. This equality was proved in \cite{Remarkable}, Proposition~3.5, in the context of a semigroup of homomorphisms with respect to multiplicative free convolution.

\end{remark}

\subsection{The top Lyapunov exponent}

The constant $\gamma$ defined in Theorem \ref{thm:matrixperp} is known as the top Lyapunov exponent, a fundamental object in understanding the behavior of many dynamical systems. Orthogonal invariance often greatly simplifies the analysis of matrix products; in the present setting, the tools developed above lead to an elementary proof of a known formula for the top Lyapunov exponent.

\begin{lemma}\label{lem:lyap} 
Let $\gamma$ be the top Lyapunov exponent defined in \eqref{eq:topL}. 
    If $\ub\Ab\Ab^\top \ub^\top\stackrel{d}{=}\Ab\Ab^\top$ for any orthogonal matrix $\ub$, then 
\[
\gamma = \E[\log(\Ab \Ab^\top)_{11}].
\]
\end{lemma}
\begin{remark}
    In \cite{Newman86} (see also \cite[Appendix C]{ABO}), the author considered orthogonal invariance of the form $\ub\Ab^\top\Ab \ub^\top\stackrel{d}{=}\Ab^\top\Ab$ for any orthogonal matrix $\ub$ and proved that in such case 
    \[
\gamma = \E[\log(\Ab^\top \Ab)_{11}].
\]
We propose a new, elementary method to prove such result. 
\end{remark}
\begin{proof}[Proof of Lemma~\ref{lem:lyap}]


Assume that ${\bf m}\in\mathrm{Sym}_{\geq0}(N)$ with eigenvalues $\lambda_i({\bf m})$ for $i=1,\ldots,N$. Then,
\[
{\bf m}_{11} = e_1^\top  {\bf m} e_1\leq \max_{x\in\R^N\colon x^\top x =1} x^\top {\bf m} x =\lambda_{\max}({\bf m}).
\]

Assume that ${\bf M}$ is orthogonally invariant. Then ${\bf M}\stackrel{d}{=} \mathbf{U}{\bf \Lambda}\mathbf{U}^\top$, where ${\bf \Lambda}$ is a diagonal matrix of eigenvalues of $M$, with ${\bf \Lambda}_{ii}=\lambda_i({\bf M})$ with $\lambda_i({\bf M})\geq\lambda_{i+1}({\bf M})$, $\mathbf{U}$ is a Haar orthogonal matrix and $\mathbf{U}$ and ${\bf \Lambda}$ are independent. 
Moreover,  
\[
{\bf M}_{11} = e_1^\top  {\bf M} e_1 \stackrel{d}{=} (\mathbf{U}^\top e_1)^\top {\bf \Lambda} \,\mathbf{U}^\top  e_1 = \sum_{j=1}^N \lambda_j({\bf M}) \mathbf{U}_{1j}^2\geq \lambda_{\max}({\bf M}) \mathbf{U}_{11}^2.
\]
Thus, 
\begin{align}\label{eq:3.12}
\E[\log {\bf M}_{11}]\leq \E[\log \lambda_{\max}({\bf M})]\leq \E[\log {\bf M}_{11} -\log \mathbf{U}_{11}^2],
\end{align}
where $\mathbf{U}^2_{11}\sim \mathrm{Beta}(1/2,(N-1)/2)$ and so $\E[\log \mathbf{U}_{11}^2] = \Psi(1/2)-\Psi(N/2)$, where $\Psi$ is the digamma function. 

Recall that $\gamma=\inf_{n\geq 1} n^{-1}\E\left[ \log( \lambda_{\max}\left(\Pi_n^\top \Pi_n \right))\right]$, where $\Pi_n = \Ab_n\Ab_{n-1}\ldots\Ab_1$. Since $\lambda_{\max}(\ab{\bf b}) = \lambda_{\max}({\bf b} \ab)$, we have 
\begin{align}\label{eq:newdefgamma}
\gamma = \inf_{n\geq 1} n^{-1}\E\left[ \log( \lambda_{\max}\left( {\bf M}_n \right))\right],\qquad \mbox{ with }{\bf M}_n= \Pi_n\Pi_n^\top.
\end{align}

Since ${\bf M}_n=\Ab_n{\bf M}_{n-1}\Ab_n^\top$ with ${\bf M}_1=\Ab_1\Ab_1^\top$, we prove by induction that ${\bf M}_n$ is orthogonally invariant for each $n$. The case $n=1$ follows from the assumption on $\Ab_1\Ab_1^\top$.
Now assume that ${\bf M}_{n-1}$ is orthogonally invariant. Write the polar decomposition (it is well defined by Lemma \ref{lem:measurable})
\[
\Ab_n = |\Ab_n|\,\Ub_{\Ab_n},
\qquad |\Ab_n| := (\Ab_n\Ab_n^\top)^{1/2},
\]
where $\Ub_{\Ab_n}$ is orthogonal. Then for any deterministic orthogonal matrix $\ub$,
\begin{align*}
\ub {\bf M}_n \ub^\top
&= \ub |\Ab_n|\,\Ub_{\Ab_n} {\bf M}_{n-1} \Ub_{\Ab_n}^\top |\Ab_n| \ub^\top \\
&= (\ub |\Ab_n|\ub^\top)\,(\ub \Ub_{\Ab_n}) {\bf M}_{n-1} (\ub \Ub_{\Ab_n})^\top\,(\ub |\Ab_n|\ub^\top) \\
&\stackrel{d}{=} (\ub |\Ab_n|\ub^\top)\,{\bf M}_{n-1}\,(\ub |\Ab_n|\ub^\top),
\end{align*}
since ${\bf M}_{n-1}$ is orthogonally invariant and independent of $\Ab_n$, hence also independent of $\Ub_{\Ab_n}$.
Next,
\[
\ub |\Ab_n| \ub^\top
= \ub (\Ab_n\Ab_n^\top)^{1/2}\ub^\top
= (\ub \Ab_n\Ab_n^\top \ub^\top)^{1/2}.
\]
Therefore, by the orthogonal invariance of $\Ab_n\Ab_n^\top$, we have $\ub |\Ab_n| \ub^\top \stackrel{d}{=} |\Ab_n|$.
Combining this with the independence of ${\bf M}_{n-1}$ and $\Ab_n$, we obtain
\[
\ub {\bf M}_n \ub^\top
\stackrel{d}{=} |\Ab_n|\,{\bf M}_{n-1}\,|\Ab_n| \stackrel{d}{=}{\bf M}_n,
\]
so ${\bf M}_n$ is orthogonally invariant.

Then, we apply Corollary \ref{prop:39} with $k=1$ successively to obtain 
\[
({\bf M}_n)_{11}\stackrel{d}{=}\prod_{i=1}^n (\Ab_i \Ab_i^\top)_{11}. 
\]
Thus we get
\[\E[\log ({\bf M}_n)_{11}]=n\,\E[\log (\Ab \Ab^\top)_{11}]\]
and the assertion follows from \eqref{eq:newdefgamma} and  inequalities in \eqref{eq:3.12} applied to ${\bf M}={\bf M}_n$.

\end{proof}

\begin{remark}\label{rem:rank-one}
  The $(1,1)$-entry function in Lemma \ref{lem:lyap} can be replaced by the trace of any rank-one projection.
    Indeed, any rank-one projection $P$ can be expressed as $(\ub^\top e_1)(\ub^\top e_1)^\top$, where $e_1=(1,0\ldots,0)^\top\in\R^N$ and $\ub$ is an orthogonal matrix. Then, we have for orthogonally invariant $\Xb=\Ab\Ab^\top$, 
    \[
    \Xb_{11}\stackrel{d}{=}( \ub \Xb \ub^\top)_{11} = e_1^\top ( \ub \Xb \ub^\top)e_1  = \mathrm{Tr}( (\ub^\top e_1)(\ub^\top e_1)^\top \Xb) = \mathrm{Tr}(P\Xb).
    \]
\end{remark}

\begin{example}\label{ex:example_gamma}
Assume that $\Ab\sim \mathcal{B}'_{\alpha,\alpha+\beta}(N)$ with $\alpha,\beta>(N-1)/2$, as in Example \ref{ex:1}. We have 
$\Ab_{11}\sim\mathcal{B}'_{\alpha,\alpha+\beta-\frac{N-1}{2}}(1)$ (the classical univariate Beta-prime distribution), as shown, e.g., in \cite[Section B.2]{ABO}. From this, we obtain
\[
\gamma = \E[\log(\Ab_{11})]=\Psi(\alpha)-\Psi\left(\alpha+\beta-\tfrac{N-1}{2}\right),
\]
where $\Psi(x)=\frac{\dd}{\dd x}\log(\Gamma(x))$ is the digamma function. Since $\Psi$ is strictly increasing, we conclude that $\gamma<0$.
\end{example}

\subsection{Matrix perpetuities under orthogonal invariance}

In this subsection we note that if $\Xb$ is a solution to \eqref{eq:matperp} of size $N\times N$, then sub-matrices $\Xb^{[k]}$ of size $k\times k$ are themselves solutions to \eqref{eq:matperp}, with truncated $\Ab,\Bb$.
Let us note some further properties of the solution to \eqref{eq:matperp}. We start with an analog of Proposition \ref{prop:390}.

\begin{lemma}\label{lem:basic_prop}
Assume that $(\Ab,\Bb)$ satisfy \eqref{eq:log_moments} and \eqref{eq:topL}. Let $\Xb$ be the unique solution to \eqref{eq:matperp}. If
\[
(\ub \Ab\Ab^\top \ub^\top,\ub \Bb \ub^\top)\stackrel{d}{=}(\Ab\Ab^\top,\Bb)
\qquad\text{for every orthogonal matrix }\ub,
\]
then $\Xb$ is orthogonally invariant.
In such case, $\Xb$ also satisfies  
\[
\Xb\stackrel{d}{=} m(\Ab)\,\Xb\,m(\Ab)^\top+\Bb,
\qquad (m(\Ab),\Bb)\ \text{and }\Xb\ \text{are independent}.
\]
\end{lemma}

\begin{proof}
Let $(\Ab_n,\Bb_n)_{n\ge1}$ be iid copies of $(\Ab,\Bb)$, independent of $\Xb_0$, and define a Markov chain
\[
\Xb_n=\Ab_n \Xb_{n-1}\Ab_n^\top+\Bb_n,\qquad n\ge1.
\]
Then, by Theorem \ref{thm:matrixperp}, 
\[
\Xb_n \stackrel{d}{\longrightarrow} \Xb,
\]
where $\Xb$ is the unique solution to \eqref{eq:matperp}.

Since the map $\ab\mapsto \ab^{1/2}$ is equivariant under orthogonal conjugation on the cone of symmetric positive definite matrices, for every orthogonal matrix $\ub$,
\[
\ub |\Ab_n| \ub^\top
= \bigl(\ub \Ab_n\Ab_n^\top \ub^\top\bigr)^{1/2}.
\]
Hence
\[
(\ub |\Ab_n| \ub^\top,\ub \Bb_n \ub^\top)\stackrel{d}{=}(|\Ab_n|,\Bb_n).
\]

Assume that $\Xb_0$ is orthogonally invariant, e.g., $\Xb_0=0$. 
We will show by induction that $\Xb_n$ is orthogonally invariant for every $n\ge1$.
Assume now that $\Xb_{n-1}$ is orthogonally invariant.
By Lemma \ref{lem:measurable}, there exists a $\sigma(\Ab_n)$-measurable orthogonal matrix $\Ub_n$ such that 
\[
\Ab_n=|\Ab_n| \Ub_n.
\]
Because $\Xb_{n-1}$ is independent of $(\Ab_n,\Bb_n)$ and $\Ub_n$ is $\sigma(\Ab_n)$-measurable, by orthogonal invariance of $\Xb_{n-1}$, we obtain
\[
(|\Ab_n|, \Bb_n, \Ub_n \Xb_{n-1}\Ub_n^\top) \stackrel{d}{=} (|\Ab_n|, \Bb_n, \Xb_{n-1}).
\]
It follows that
\[
\Xb_n
=\Ab_n \Xb_{n-1}\Ab_n^\top+\Bb_n
=|\Ab_n|\Ub_n \Xb_{n-1}\Ub_n^\top |\Ab_n|+\Bb_n
\stackrel{d}{=} |\Ab_n| \Xb_{n-1}|\Ab_n|+\Bb_n.
\]
Now let $\ub$ be any fixed orthogonal matrix. Using the orthogonal invariance of $\Xb_{n-1}$, the independence of $\Xb_{n-1}$ and $(|\Ab_n|,\Bb_n)$, and the invariance of $(|\Ab_n|,\Bb_n)$, we obtain
\begin{align*}
\ub \Xb_n \ub^\top
&\stackrel{d}{=}
\ub\bigl(|\Ab_n|\Xb_{n-1}|\Ab_n|+\Bb_n\bigr)\ub^\top\\
&=
(\ub|\Ab_n|\ub^\top)(\ub\Xb_{n-1}\ub^\top)(\ub|\Ab_n|\ub^\top)+\ub\Bb_n\ub^\top\\
&\stackrel{d}{=}
|\Ab_n| \Xb_{n-1}|\Ab_n|+\Bb_n
\stackrel{d}{=}\Xb_n.
\end{align*}
Thus $\Xb_n$ is orthogonally invariant for every $n$.

Since $\Xb_n\stackrel{d}{\longrightarrow} \Xb$, orthogonal invariance passes to the limit, and therefore
\[
\ub \Xb \ub^\top \stackrel{d}{=} \Xb
\qquad\text{for every orthogonal matrix }\ub.
\]

We now prove the second assertion. By Lemma \ref{lem:m_polar}, there exists a $\sigma(\Ab)$-measurable orthogonal matrix $\Vb_{\Ab}$ such that 
\[
\Ab = m(\Ab)\Vb_{\Ab}.
\]
Since $\Xb$ is orthogonally invariant and independent of $(\Ab,\Bb)$, we have
\[
(\Ab,\Bb,\Vb_{\Ab}\Xb \Vb_{\Ab}^\top)\stackrel{d}{=}(\Ab,\Bb,\Xb).
\]
Using \eqref{eq:matperp}, we conclude that
\begin{align*}
\Xb
&\stackrel{d}{=} \Ab \Xb \Ab^\top+\Bb=m(\Ab)\,\Vb_{\Ab}\Xb \Vb_{\Ab}^\top\,m(\Ab)^\top+\Bb\stackrel{d}{=} m(\Ab)\,\Xb\,m(\Ab)^\top+\Bb.
\end{align*}
\end{proof}

Next we note that if $\Xb$ is a perpetuity (i.e. is a solution to \eqref{eq:matperp}), then the whole family $(\Xb^{[k]})_{1\leq k\leq N}$ consists of perpetuities.
\begin{lemma}\label{lem:subperp}
Under assumptions of Lemma \ref{lem:basic_prop}, for any $1\leq k\leq N$ we have
\[
 \Xb^{[k]}\stackrel{d}{=}((\Ab\Ab^\top)^{[k]})^{1/2}\Xb^{[k]} ((\Ab\Ab^\top)^{[k]})^{1/2}+\Bb^{[k]},\qquad ((\Ab\Ab^\top)^{[k]},\Bb^{[k]})\mbox{ and }\Xb^{[k]}\mbox{ are independent.}
\]
In particular \begin{align}\label{eq:11}
 \Xb_{11}\stackrel{d}{=} (\Ab\Ab^\top)_{11}\Xb_{11}+\Bb_{11},\qquad ((\Ab\Ab^\top)_{11}, \Bb_{11})\mbox{ and }\Xb_{11}\mbox{ are independent}.
 \end{align}

\end{lemma}

\begin{proof}
Let $m(\Ab) = T_{\Ab\Ab^\top}$ be the lower triangular matrix from the Cholesky decomposition of $\Ab\Ab^\top=T_{\Ab\Ab^\top}T_{\Ab\Ab^\top}^\top$. Since $(T \xb T^\top)^{[k]} = T^{[k]} \xb^{[k]} (T^{[k]})^\top$ for any lower triangular matrix $T$ and any matrix $\xb$, by Lemma \ref{lem:basic_prop}, we obtain 
\[
 \Xb^{[k]}\stackrel{d}{=}T_{\Ab\Ab^\top}^{[k]} \Xb^{[k]} (T_{\Ab\Ab^\top}^{[k]})^\top+\Bb^{[k]},\qquad (T_{\Ab\Ab^\top}^{[k]},\Bb^{[k]})\mbox{ and }\Xb^{[k]}\mbox{ are independent.}
\]
Now, note that $T_{\Ab\Ab^\top}^{[k]} (T_{\Ab\Ab^\top}^{[k]})^\top = (\Ab\Ab^\top)^{[k]} = T_{(\Ab\Ab^\top)^{[k]}} T_{(\Ab\Ab^\top)^{[k]}}^\top$. Therefore $T_{\Ab\Ab^\top}^{[k]} = T_{(\Ab\Ab^\top)^{[k]}}$.

Since orthogonal invariance of $(\Ab\Ab^\top,\Bb)$ implies orthogonal invariance of $((\Ab\Ab^\top)^{[k]},\Bb^{[k]})$, again by Lemma \ref{lem:basic_prop}, we can replace $T_{\Ab\Ab^\top}^{[k]}$ by $((\Ab\Ab^\top)^{[k]})^{1/2}$. 
\end{proof}

Assuming orthogonal invariance, we obtain a simplified version of Theorem~\ref{thm:matrixperp} under particularly simple assumptions.

\begin{Corollary}\label{cor:matrixperp_inv}
Assume that
\[
(\ub \Ab\Ab^\top \ub^\top,\ub \Bb \ub^\top)\stackrel{d}{=}(\Ab\Ab^\top,\Bb)
\qquad\text{for every orthogonal matrix }\ub.
\]
If
\begin{align*}
\gamma=\E[\log(\Ab \Ab^\top)_{11}] <0\quad\mbox{and}\quad \E\left[ \log^+(\Bb\Bb^\top)_{11}\right]<+\infty 
\end{align*}
then there exists a unique solution to \eqref{eq:matperp}. 
\end{Corollary}
\begin{proof}
    We verify the assumptions of Theorem~\ref{thm:matrixperp}. By Lemma \ref{lem:lyap} we obtain the representation of the top Lyapunov exponent. 

    It remains to check the logarithmic moment assumptions in
\eqref{eq:log_moments}. Let \(\mathbf S\) be an orthogonally invariant
positive semidefinite random matrix, and let \(\Ub\) be an independent Haar
orthogonal matrix. If \(\Ub \mathbf S \Ub^\top \stackrel{d}{=} \mathbf S\),
    then repeating the steps from the proof of Lemma \ref{lem:lyap} we obtain (here $\stackrel{d}{\leq}$ denotes the usual stochastic order)
    \[
\lambda_{\max}(\mathbf{S})\Ub_{11}^2 \stackrel{d}{\leq} \mathbf{S}_{11} \leq \lambda_{\max}(\mathbf{S}),
    \]
    where $\Ub_{11}^2\sim\mathrm{Beta}(1/2,(N-1)/2)$.
It follows that
\[
\E\left[\log^+\lambda_{\max}(\mathbf S)\right]<\infty
\quad\Longleftrightarrow\quad
\E\left[\log^+\mathbf S_{11}\right]<\infty.
\]
Applying this to \(\mathbf S=\Ab\Ab^\top\) and \(\mathbf S=\Bb\Bb^\top\), we see
that \eqref{eq:log_moments} is equivalent to
\[
\E\!\left[\log^+(\Ab\Ab^\top)_{11}\right]<\infty
\qquad\text{and}\qquad
\E\!\left[\log^+(\Bb\Bb^\top)_{11}\right]<\infty.
\]
The second condition is assumed. The first one follows from $\gamma<0$. 

Thus all assumptions of Theorem~\ref{thm:matrixperp} are satisfied, and
\eqref{eq:matperp} admits a unique solution.
\end{proof}

\section{Tails of matrix perpetuities}\label{sec:tails}

The problem of studying tail asymptotics for solutions of multivariate affine fixed-point equations has a long history.

One of the most celebrated results in this area is Kesten’s theorem. Consider the vector perpetuity equation
\begin{align*}
\underline{X} \stackrel{d}{=} \tilde{\Ab}\underline{X}+\underline{B},\qquad (\tilde{\Ab},\underline{B})\ \text{and}\ \underline{X}\ \text{are independent},
\end{align*}
where $\underline{X},\underline{B}\in\R^N$ and $\tilde{\Ab}\in\mathrm{Mat}(N)$.

For several variants of Kesten’s theorem, see \cite[Section 4.4]{BDM} as well as \cite{Kes73, GLP16}. A key ingredient common to these variants is the function $\tilde h$ defined by
    \[
    \tilde{h}(s) = \inf_{n\geq 1}\E[ \| \tilde{\Ab}_1\cdots \tilde{\Ab}_n\|^s]^{1/n},
    \]
where $(\tilde{\Ab}_n)_{n\geq 1}$ are iid copies of $\tilde{\Ab}$.

While the precise conclusions of Kesten’s theorem depend on the specific setting (see, e.g., \cite[Theorems 4.4.5, 4.4.15, 4.4.18]{BDM}), they share a common theme: $X$ exhibits power-law tails of order $t^{-\eta}$, where $\eta>0$ is the unique solution of $\tilde h(\eta)=1$. For instance, Guivarc'h and Le Page showed that, under additional assumptions (see \cite[Theorem 5.2]{GLP16} or \cite[Theorem 4.4.18]{BDM}), there exists a nonzero Radon measure $\mu$ on $(\R\cup\{-\infty,+\infty\})^{N}\setminus\{0\}$ such that
\[
t^{\eta}\,\P\left(t^{-1}\underline{X}\in \cdot\right) \stackrel{v}{\longrightarrow}\mu,\qquad t\to+\infty,
\]
where $\stackrel{v}{\longrightarrow}$ denotes the vague convergence.

Our interest, however, is not in the asymptotic behavior of $\P(\underline{X}\in\cdot)$, but rather in the tail asymptotics of the expected empirical spectral distribution $\mu_\Xb(\cdot)$ (recall Definition \ref{def:spect}). We assume  orthogonal invariance of the law of $(\Ab,\Bb)$, which is natural in our framework and considerably simplifies the analysis. In our framework condition $\tilde h(\eta)=1$ is replaced by the much simpler moment condition $\E[\Ab_{11}^{\eta}]=1$. In addition, our approach yields tail asymptotics for $\mu_{\Xb}$, as well as for $\P_{\lambda_{\max}(\Xb)}$ and $\P_{\Xb_{11}}$, with explicit positive constants.

\begin{thm}\label{thm:tailsmat}
  Assume that $N\geq 2$ and
    \begin{enumerate}
    \item[(i)] $\Ab,\Bb\in\mathrm{Sym}_{\geq0}(N)$ a.s.,
        \item[(ii)] $(\ub\Ab \ub^\top, \ub \Bb \ub^\top)\stackrel{d}{=}(\Ab,\Bb)$ for any orthogonal matrix $\ub$,
        \item[(iii)] there exists $\eta>0$ such that 
        \begin{enumerate}
            \item $\E[\Ab_{11}^\eta]=1$,
             \item $\E[\Ab_{11}^\eta \log^+(\Ab_{11})]<\infty$,
             \item $\E[\Bb_{11}^\eta]<\infty$,
             \item $\E[ \det( \Ab^{[2]})^{\eta/2}]<1$,
        \end{enumerate}
        \item[(iv)] the conditional law of $\log \Ab_{11}$ given $\{\Ab_{11}\neq 0\}$ is non-arithmetic,
        \item[(v)] $\P(\Ab_{11}x+\Bb_{11}=x)<1$ for $x\in\R$.
    \end{enumerate}
   Then, there exists a unique solution $\Xb$ to
\begin{align}\label{eq:perpMatS}
\Xb\stackrel{d}{=}\Ab^{1/2} \Xb \Ab^{1/2} +\Bb,\qquad (\Ab,\Bb)\mbox{ and }\Xb\mbox{ are independent},
\end{align} 
which satisfies 
\[
\mu_\Xb\big( (t,\infty)\big) \sim \frac{1}{N}\P(\lambda_{\max}(\Xb) > t) \sim \frac1N \frac{\sqrt{\pi}\, \Gamma\left(\eta+\dfrac{N}{2}\right)}{\Gamma\left(\dfrac{N}{2}\right)\, \Gamma\left(\eta+\tfrac12\right)} \P(\Xb_{11}>t)
\]
and
    \begin{align}\label{eq:limit11}
\lim_{t\to+\infty} t^\eta\, \P(\Xb_{11}>t) = \frac{ \E\left[ (\Ab_{11}\Xb_{11}+\Bb_{11})^\eta - (\Ab_{11}\Xb_{11})^\eta \right]}{\eta \,\E[\Ab_{11}^\eta \log\Ab_{11} ]}\in(0,+\infty).
    \end{align}
\end{thm}

The proof of Theorem \ref{thm:tailsmat} is given in Section \ref{sec:proof_tails}. It relies on several auxiliary lemmas, which are stated and proved in Section \ref{lem:Lemmas}.

\subsection{Auxiliary lemmas}\label{lem:Lemmas}

For a symmetric matrix $\ab$, let
\[
\lambda_1(\ab)\geq \lambda_2(\ab)\geq \cdots \geq \lambda_N(\ab)
\]
denote its eigenvalues in non-increasing order. 

\begin{defin}
Let $(\Ab_n)_{n\ge 1}$ be iid copies of $\Ab$ and set
\[
\Pi_n:=\Ab_n\cdots \Ab_1,\qquad {\bf M}_n:=\Pi_n \Pi_n^\top\in\mathrm{Sym}_{\geq0}(N).
\]
Fix $k\in\{1,\ldots,N\}$ and define 
\begin{align}\label{eq:def_h}
h_k(s):=    \inf_{n\in\mathbb{N}} \E\left[ \prod_{i=1}^k \lambda_i({\bf M}_n)^{s/k}\right]^{1/n}
\end{align}
for nonnegative $s$ such that $ \E\left[ \prod_{i=1}^k \lambda_i(\Ab \Ab^\top)^{s/k}\right]<\infty$.
\end{defin}

\begin{remark}\label{rem:submult}
For each $k$ the quantity 
\[
\lambda_{[1:k]}({\bf m}):= \prod_{i=1}^k \lambda_i({\bf m})
\]
agrees with
\[
\lambda_{\max}(\wedge^k {\bf m}),
\]
where $\wedge^k$ denotes the $k$th exterior power. We will not make explicit use of exterior powers in the statements or proofs below, but this identification implies in particular that $\lambda_{[1:k]}$ is submultiplicative. This justifies in particular why $h_k(s)$ is well defined provided that $ \E\left[ \prod_{i=1}^k \lambda_i(\Ab \Ab^\top)^{s/k}\right]<\infty$.
\end{remark}

Our results will only use $h_1$ and $h_2$, but we consider general $h_k$ for completeness.

\begin{lemma}\label{lem:hi_invariance}
Assume that $\ub\Ab\Ab^\top\,\ub^\top\stackrel{d}{=}\Ab\Ab^\top$ for every orthogonal matrix $\ub$. Fix $k\in\{1,\ldots,N\}$ and let $s>0$ be such that
$ \E\left[ \prod_{i=1}^k \lambda_i(\Ab\Ab^\top)^{s/k}\right]<\infty$.
Then
\begin{align}\label{eq:hi_representation}
h_k(s)=\E\left[\det\big((\Ab\Ab^\top)^{[k]}\big)^{s/k}\right].
\end{align}
\end{lemma}
\begin{proof}
Let $n\ge 1$. By Corollary~\ref{prop:39} iterated
along the product, one has the distributional identity 
\begin{align}\label{eq:minor_product_i}
({\bf M}_n)^{[k]} \stackrel{d}{=}\ 
\big(\tilde{\Ab}_n^{[k]}\big)^{1/2}\cdots \big(\tilde{\Ab}_{2}^{[k]}\big)^{1/2}
(\tilde{\Ab}_1^{[k]})
\big(\tilde{\Ab}_{2}^{[k]}\big)^{1/2}\cdots \big(\tilde{\Ab}_n^{[k]}\big)^{1/2},
\end{align}
where $\tilde{\Ab}_i := \Ab_i\Ab_i^\top$.
Taking determinants on both sides gives
\begin{align}\label{eq:wedge_entry_product_i}
\det\big({\bf M}_n^{[k]}\big)
\stackrel{d}{=} \prod_{i=1}^n \det\big(\tilde{\Ab}_i^{[k]}\big).
\end{align}
By the Cauchy interlacing inequalities, we have for any ${\bf m}\in\mathrm{Sym}(N)$ and $i=1,\ldots,k$,
\[
\lambda_i({\bf m}) \geq \lambda_i({\bf m}^{[k]}).
\]
Thus, since ${\bf M}_n$ is positive semidefinite, we obtain 
\[
\prod_{i=1}^k \lambda_i({\bf M}_n) \geq \prod_{i=1}^k \lambda_i({\bf M}_n^{[k]}) =  \det\big({\bf M}_n^{[k]}\big).
\]
Therefore, by \eqref{eq:wedge_entry_product_i} and independence,
\[
\E\left[\prod_{i=1}^k \lambda_i({\bf M}_n)^{s/k}\right]
\ge \E\left[\det\big(\tilde{\Ab}^{[k]}\big)^{s/k}\right]^n = 
\E\left[\det\big((\Ab\Ab^\top)^{[k]}\big)^{s/k}\right]^n.
\]
Taking $1/n$-th powers and then the infimum over $n$ yields
\begin{align}\label{eq:hi_lower}
h_k(s)\ge \E\left[\det\big((\Ab\Ab^\top)^{[k]}\big)^{s/k}\right].
\end{align}
For the reverse inequality, observe that for any positive semidefinite matrix \({\bf m}\), 
\begin{align}\label{eq:ineq+identity}
\prod_{i=1}^k \lambda_i({\bf m})
\le e_k({\bf m})
= \sum_{1\le j_1<\cdots<j_k\le N}
\det\bigl({\bf m}^{[j_1,\ldots,j_k]}\bigr),
\end{align}
where \(e_k({\bf m})\) denotes the \(k\)-th elementary symmetric polynomial in the eigenvalues of \({\bf m}\), and \({\bf m}^{[j_1,\ldots,j_k]}\) denotes the principal \(k\times k\) submatrix of \({\bf m}\) indexed by \(\{j_1,\ldots,j_k\}\), \cite[Theorem 1.2.16]{MAnal}. We have
\[
\Big(\sum_{r=1}^D x_r\Big)^s \le c(D,s)\sum_{r=1}^D x_r^s,\qquad x_r\ge 0,
\]
for a constant $c(D,s)\in(0,\infty)$. In our case we set $D:=\binom{N}{k}$ and we have
\[
\prod_{i=1}^k \lambda_i({\bf M}_n)^{s/k}
\le c(D,s/k)\sum_{1\le j_1<\cdots<j_k\le N}
\det\bigl({\bf M}_n^{[j_1,\ldots,j_k]}\bigr)^{s/k}.
\]
By orthogonal invariance of $\Ab\Ab^\top$, the law of ${\bf M}_n=\Pi_n\Pi_n^\top$ is orthogonally invariant.
In particular, the distribution of ${\bf M}_n^{[j_1,\ldots,j_k]}$ does not depend on the choice of $\{j_1,\ldots,j_k\}$. Thus, 
\[
\E\left[\prod_{i=1}^k \lambda_i({\bf M}_n)^{s/k}\right]
\le c(D,s/k)\,D\ \E\left[\det\big({\bf M}_n^{[k]}\big)^{s/k}\right]
= c(D,s/k)\,D\ \E\left[\det\big((\Ab\Ab^\top)^{[k]}\big)^{s/k}\right]^n,
\]
where the last equality follows from \eqref{eq:wedge_entry_product_i}. 
Therefore,
\[
\E\left[\prod_{i=1}^k \lambda_i({\bf M}_n)^{s/k}\right]^{1/n}
\le \big(c(D,s/k)\,D\big)^{1/n}\ \E\left[\det\big((\Ab\Ab^\top)^{[k]}\big)^{s/k}\right].
\]
Letting $n\to\infty$ gives
\begin{align}\label{eq:hi_upper}
h_k(s)\le \E\left[\det\big((\Ab\Ab^\top)^{[k]}\big)^{s/k}\right].
\end{align}
Combining \eqref{eq:hi_lower} and \eqref{eq:hi_upper} yields \eqref{eq:hi_representation}.
\end{proof}

\begin{lemma}\label{lem:moment_lambda1}
Let $\Xb$ be a random matrix satisfying
\begin{equation}\label{eq:sfp_sym}
\Xb \stackrel{d}{=} \Ab\,\Xb\,\Ab^\top+\Bb,
\qquad (\Ab,\Bb)\ \text{independent of }\Xb,
\end{equation}
where $\Bb\in\mathrm{Sym}_{\geq0}(N)$ a.s. 

Assume that there exists $p>0$ such that 
\[
\E[\mathrm{Tr}( \Ab\Ab^\top)^{p}]<\infty,\quad \E\big[\mathrm{Tr}(\Bb)^{p}\big]<\infty\quad\mbox{and}\quad h_1(p)<1.
\]
Then
\[
\E\Big[\lambda_1(\Xb)^p\Big]<\infty.
\]
\end{lemma}
\begin{proof}
    The result follows from \cite[Remark 4.4.3]{BDM}, which establishes that vector
    \[
    \mathrm{vec}(\Xb)\stackrel{d}{=}(\Ab\otimes\Ab)\mathrm{vec}(\Xb)+\mathrm{vec}(\Bb),
    \]
    has $p$-th moment finite under the assumptions 
    \[
 \tilde{h}(p)<1   \quad \mbox{and}\quad \E[ |\mathrm{vec}(\Bb)|^p ]<\infty,
    \]
    where 
    \[
    \tilde{h}(p) = \inf_{n\geq 1}\E[ \| \tilde{\Ab}_1\cdots \tilde{\Ab}_n\|^p]^{1/n},\qquad \tilde{\Ab}_i=(\Ab_i\otimes\Ab_i).
    \]
We have $h_1(p) = \tilde{h}(p)$, cf. the proof of Theorem \ref{thm:matrixperp}. 
\end{proof}

Let $e_2(\xb)$ be elementary symmetric polynomial in the eigenvalues of $\xb$, i.e.,
\begin{equation}\label{eq:e2_trace_identity}
e_2(\xb)=\frac12\Big(\mathrm{Tr}(\xb)^2-\mathrm{Tr}(\xb^2)\Big).
\end{equation}

\begin{lemma}
    Let $\xb,\yb\in\mathrm{Sym}_{\geq0}(N)$ and ${\bf m}\in\mathrm{Mat}(N)$, $N\geq 2$, the following deterministic inequalities hold:
\begin{equation}\label{eq:e2_le_tr2}
0\le e_2(\xb)\le \frac12\,\mathrm{Tr}(\xb)^2,
\end{equation}
\begin{equation}\label{eq:e2_add}
e_2(\xb+\yb)\le e_2(\xb)+e_2(\yb)+\mathrm{Tr}(\xb)\mathrm{Tr}(\yb),
\end{equation}
\begin{equation}\label{eq:tr_congruence}
\mathrm{Tr}({\bf m} \xb {\bf m}^\top)\le \lambda_{1}({\bf m}{\bf m}^\top)\,\mathrm{Tr}(\xb),
\end{equation}
\begin{equation}\label{eq:e2_congruence}
e_2({\bf m} \xb {\bf m}^\top)\le \lambda_{[1:2]}({\bf m}{\bf m}^\top) \, e_2(\xb).
\end{equation}

\end{lemma}
\begin{proof}
Eq. \eqref{eq:e2_le_tr2} is obvious. For \eqref{eq:e2_add}, use \eqref{eq:e2_trace_identity} and $\mathrm{Tr}(\xb\yb)\ge0$, \[ e_2(\xb+\yb) =\frac12\Big((\mathrm{Tr}(\xb)+\mathrm{Tr}(\yb))^2-\mathrm{Tr}(\xb^2)-\mathrm{Tr}(\yb^2)-2\mathrm{Tr}(\xb\yb)\Big) \le e_2(\xb)+e_2(\yb)+\mathrm{Tr}(\xb)\mathrm{Tr}(\yb). \]  
\eqref{eq:tr_congruence} is standard. For \eqref{eq:e2_congruence},  let \(\mathbf m = \ub \ab \vb^\top\) be a singular value decomposition, with
\[
\ab=\operatorname{diag}(a_1,\dots,a_N), \qquad a_1\ge a_2\ge \cdots \ge 0,
\]
and set \({\bf z} = \vb^\top \xb \vb\). Since \(\xb \geq 0\), we also have \({\bf z}\geq 0\). By orthogonal invariance of \(e_2\),
\[
e_2(\mathbf m \xb \mathbf m^\top)=e_2(\ab {\bf z}  \ab).
\]
Using the identity (recall \eqref{eq:ineq+identity})
\[
e_2({\bf m})=\sum_{1\le i<j\le N}\det\bigl({\bf m}^{[i,j]}\bigr),
\]
we obtain
\[
e_2(\ab {\bf z} \ab)
=
\sum_{i<j}\det\bigl((\ab {\bf z} \ab)^{[i,j]}\bigr)
=
\sum_{i<j}a_i^2a_j^2 \det\bigl({\bf z}^{[i,j]}\bigr).
\]
Since \({\bf z}\geq 0\), all principal \(2\times 2\) minors are nonnegative, and therefore
\[
e_2(\mathbf m \xb \mathbf m^\top)
\le
a_1^2 a_2^2
\sum_{i<j}\det\bigl({\bf z}^{[i,j]}\bigr)
=
a_1^2a_2^2\, e_2({\bf z}).
\]
Finally, \(e_2({\bf z})=e_2(\xb)\) and
\[
a_1^2=\lambda_1(\mathbf m \mathbf m^\top)   ,
\qquad
a_2^2=\lambda_2(\mathbf m \mathbf m^\top).
\]

\end{proof}

\begin{lemma}\label{lem:moment_lambda1lambda2}
Let $\Xb$ be a $\mathrm{Sym}_{\geq0}(N)$-valued random matrix satisfying \eqref{eq:sfp_sym} with $\Bb\in\mathrm{Sym}_{\geq0}(N)$ a.s. 

Assume that there exists $p>0$ such that 
\[
\E[\mathrm{Tr}(\Ab\Ab^\top)^{2p}]<\infty,\quad \E\big[\mathrm{Tr}(\Bb)^{2p}\big]<\infty,
\]
\[
h_1(p)<1
\qquad\text{and}\qquad
h_2(2p)<1.
\]
Then
\[
\E\Big[(\lambda_1(\Xb)\lambda_2(\Xb))^p\Big]<\infty.
\]
\end{lemma}

\begin{proof}
For $\xb\in\mathrm{Sym}_{\geq0}(N)$, we have 
\[
\lambda_1(\xb)\lambda_2(\xb)\leq e_2(\xb):=\sum_{1\le i<j\le N}\lambda_i(\xb)\lambda_j(\xb),
\]
hence it suffices to prove
\begin{equation}\label{eq:goal_e2}
\E\big[e_2(\Xb)^p\big]<\infty.
\end{equation}

Let $(\Ab_n,\Bb_n)_{n\ge1}$ be iid copies of $(\Ab,\Bb)$ and set
\[
\Pi_n:=\Ab_n\cdots \Ab_1,\qquad {\bf M}_n:=\Pi_n \Pi_n^\top\in\mathrm{Sym}_{\geq0}(N).
\]
For $m\ge1$ define
\[
{\bf C}_m:=\sum_{j=1}^m \Pi_{j-1}\Bb_j\Pi_{j-1}^\top \in \mathrm{Sym}_{\geq0}(N),
\qquad \Pi_0:=I.
\]
Iterating \eqref{eq:sfp_sym} yields
\begin{equation}\label{eq:block}
\Xb \stackrel{d}{=} \Pi_m\,\Xb\,\Pi_m^\top + {\bf C}_m,
\end{equation}
where $(\Pi_m,{\bf C}_m)$ is independent of $\Xb$.
Denote the right-hand side of \eqref{eq:block} by $\Xb'$. Apply \eqref{eq:e2_add} to $\xb=\Pi_m\Xb\Pi_m^\top$ and $\yb={\bf C}_m$, then use
\eqref{eq:e2_congruence} and \eqref{eq:tr_congruence} with ${\bf m}=\Pi_m$.
Writing ${\bf M}_m=\Pi_m\Pi_m^\top$, we obtain
\begin{equation}\label{eq:e2_basic}
e_2(\Xb')
\le \lambda_{[1:2]}({\bf M}_m)\, e_2(\Xb)
+e_2({\bf C}_m)
+\lambda_{1}({\bf M}_m)\,\mathrm{Tr}(\Xb)\,\mathrm{Tr}({\bf C}_m).
\end{equation}

Let
\[
c_p:=
\begin{cases}
1, & 0<p\le 1,\\
3^{p-1}, & p\ge 1,
\end{cases}
\qquad\text{so that}\qquad
(x+y+z)^p\le c_p(x^p+y^p+z^p)\ \ \forall\, x,y,z\ge0.
\]
Raising \eqref{eq:e2_basic} to the power $p$, applying truncation $\min\{\cdot, K\}$ and taking expectations on both sides gives for positive $K$,
\begin{align}\label{eq:e2_expect}
\begin{split}
\E[\min\{e_2(\Xb)^p,K\}]
\le c_p\Big(&
\E[\lambda_{[1:2]}({\bf M}_m)^p]\E[\min\{e_2(\Xb)^p,K\}]
+ \E[e_2({\bf C}_m)^p]  \\
&+ \E[\lambda_{1}({\bf M}_m)^p\,\mathrm{Tr}({\bf C}_m)^p]\ \E[\mathrm{Tr}(\Xb)^p]
\Big).
\end{split}
\end{align}
Since $h_2(2p)<1$, there exists $m_0\ge1$ such that
$\E[\lambda_{[1:2]}({\bf M}_{m_0})^p]<1$.
Using the decomposition $\Pi_{k \cdot m_0}=\widetilde\Pi_{m_0}^{(k)}\cdots \widetilde\Pi_{m_0}^{(1)}$ into iid blocks with $\tilde{\Pi}_{m_0}^{(i)}\stackrel{d}{=}\Pi_{m_0}$,
one has for all $k\ge1$ the submultiplicative bound (see Remark \ref{rem:submult})
\[
\lambda_{[1:2]}({\bf M}_{k\, m_0})
= \lambda_{[1:2]}(\Pi_{k \,m_0}\Pi_{k\, m_0}^\top)
\le \prod_{j=1}^k \lambda_{[1:2]}(\widetilde\Pi_{m_0}^{(j)}(\widetilde\Pi_{m_0}^{(j)})^\top),
\]
and hence, by independence,
\[
\E[\lambda_{[1:2]}({\bf M}_{k\, m_0})^p]\le \E[\lambda_{[1:2]}({\bf M}_{m_0})^p]^k.
\]
Choosing $k$ large enough and setting $m:=k\, m_0$ yields
\begin{equation}\label{eq:choose_m}
c_p\,\E[\lambda_{[1:2]}({\bf M}_m)^p]<\frac12.
\end{equation}
For this $m$, we can absorb the first term on the right-hand side of \eqref{eq:e2_expect} into the left, obtaining
\begin{equation}\label{eq:e2_absorb}
\E[\min\{e_2(\Xb)^p,K\}]
\le 2c_p\Big(
\E[e_2({\bf C}_m)^p]
+ \E[\lambda_{1}({\bf M}_m)^p\,\mathrm{Tr}({\bf C}_m)^p]\ \E[\mathrm{Tr}(\Xb)^p]
\Big).
\end{equation}
Thus it remains to show that the expectations on the right-hand side are finite.

The assumption $\E[\mathrm{Tr}(\Ab \Ab^\top)^{2p}]<\infty$ implies that $\E[\lambda_{1}(\Ab\Ab^\top)^{2p}]<\infty$, which, by submultiplicativity, gives $\E[\lambda_{1}({\bf M}_m)^{2p}]<\infty$ for all $m$. By definition of ${\bf C}_m$ and \eqref{eq:tr_congruence},
\[
\mathrm{Tr}({\bf C}_m)=\sum_{j=1}^m \mathrm{Tr}(\Pi_{j-1}\Bb_j\Pi_{j-1}^\top)
\le \sum_{j=1}^m \lambda_{1}({\bf M}_{j-1})\,\mathrm{Tr}(\Bb_j).
\]
Since $m$ is fixed and $\E[\mathrm{Tr}(\Bb)^{2p}]<\infty$, it follows that
$\E[\mathrm{Tr}({\bf C}_m)^{2p}]<\infty$ (using finiteness of the relevant moments of ${\bf M}_{j-1}$).
Consequently, by \eqref{eq:e2_le_tr2},
\[
\E[e_2({\bf C}_m)^p]\le 2^{-p}\E[\mathrm{Tr}({\bf C}_m)^{2p}]<\infty.
\]
Moreover, $\E[\lambda_{1}({\bf M}_m)^p\,\mathrm{Tr}({\bf C}_m)^p]<\infty$ follows from H\"older's inequality and the
already established finiteness of $\E[\lambda_{1}({\bf M}_m)^{2p}]$ and $\E[\mathrm{Tr}({\bf C}_m)^{2p}]$.

By Lemma \ref{lem:moment_lambda1}, the condition $h_1(p)<1$ along with $\E[\mathrm{Tr}({\bf C}_{m})^p]<\infty$, implies finiteness of the $p$-moment largest eigenvalue, which means that we also have
\begin{equation}\label{eq:trX_finite}
\E[\mathrm{Tr}(\Xb)^p]<\infty.
\end{equation}

All terms on the right-hand side of \eqref{eq:e2_absorb} are finite, hence for any positive $K$,  the upper bound on $\E[\min\{e_2(\Xb)^p,K\}]$ is finite and independent of $K$. Letting $K\to\infty$, we obtain 
\[
\E\big[(\lambda_1(\Xb)\lambda_2(\Xb))^p\big]\le \E[e_2(\Xb)^p]<\infty.
\]
\end{proof}

\begin{lemma}\label{lem:lambda2_littleo_from_moment}
Let $\Xb\in\mathrm{Sym}_{\geq0}(N)$ be a random matrix with ordered eigenvalues
$\lambda_1(\Xb)\ge \lambda_2(\Xb)\ge \cdots \ge 0$. Fix $\alpha>0$.
Assume that there exist constants $c>0$ and $t_0>0$ such that
\[
\P(\lambda_1(\Xb)>t) \ge c\,t^{-\alpha},\qquad \forall\,t\ge t_0,
\]
and 
\[
\E\left[(\lambda_1(\Xb)\lambda_2(\Xb))^{\alpha/2}\right]<\infty.
\]
Then
\[
\P(\lambda_2(\Xb)>t)=o\big(\P(\lambda_1(\Xb)>t)\big),\qquad \mbox{as }t\to\infty.
\]
\end{lemma}

\begin{proof}
Let $Y:=\lambda_1(\Xb)\lambda_2(\Xb)\ge 0$. Since $\lambda_1(\Xb)\ge \lambda_2(\Xb)$,
\[
\{\lambda_2(\Xb)>t\}\subseteq \{\lambda_1(\Xb)\lambda_2(\Xb)>t^2\}=\{Y>t^2\}.
\]
Therefore, for all $t>0$,
\[
\P(\lambda_2(\Xb)>t)\le \P(Y>t^2).
\]
Since $\E[Y^{\alpha/2}]<\infty,$ we have $\lim_{t\to \infty} t^{\alpha/2}\P(Y>t) = 0$. Thus,
\[
\frac{\P(\lambda_2(\Xb)>t)}{\P(\lambda_1(\Xb)>t)} \leq \frac{ \P(Y>t^2)}{c\, t^{-\alpha}} = c^{-1} (t^2)^{\alpha/2}\P(Y>t^2) \to 0.
\]
\end{proof}

\subsection{Proof of main theorem}\label{sec:proof_tails}

\begin{proof}[Proof of Theorem \ref{thm:tailsmat}]

The existence and uniqueness follows from Corollary \ref{cor:matrixperp_inv}, by applying Jensen's inequality and domination of $\log(x)$ by $x^\alpha$. 

By (iii-a) and (iii-d) and Lemma \ref{lem:hi_invariance}, we have 
\[
h_1(\eta)=1\quad\mbox{and} \quad h_2(\eta)<1. 
\]
Since $h_1$ is a strictly convex function with $h_1(0)=1=h_1(\eta)$, we have $h_1(\eta/2)<1$. Moreover, $\E[\mathrm{Tr}(\Ab)^\eta]\leq \max\{1, N^{\eta-1}\} \sum_{i=1}^N\E[\Ab_{ii}^\eta]<\infty$ and similarly $\E[\mathrm{Tr}(\Bb)^\eta]<\infty$. Thus, by Lemma \ref{lem:moment_lambda1lambda2} applied with $p=\eta/2$, we obtain 
\[
\E[(\lambda_1(\Xb)\lambda_2(\Xb))^{\eta/2}]<\infty. 
\]
By Lemma \ref{lem:subperp}, 
\[
 \Xb_{11}\stackrel{d}{=} \Ab_{11}\Xb_{11}+\Bb_{11},\qquad (\Ab_{11}, \Bb_{11})\mbox{ and }\Xb_{11}\mbox{ are independent}.
\]
Under (iii-a)-(iii-c) and (iv)-(v), the univariate tail result \cite[Theorem 4.1]{Gol91}, gives the tail asymptotics \eqref{eq:limit11} of $\Xb_{11}$. 

Since $\lambda_{1}(\Xb)\geq \Xb_{11}$, we have 
\[
\P(\lambda_{1}(\Xb)>t)\geq \P(\Xb_{11}>t) \sim c\, t^{-\eta}. 
\]
By Lemma \ref{lem:lambda2_littleo_from_moment}, 
\[
\P(\lambda_2(\Xb)>t) = o\big(\P(\lambda_1(\Xb)>t)\big). 
\]
In particular,
\[
\mu_\Xb\big( (t,\infty)\big) = \frac{1}{N}\sum_{k=1}^N \P(\lambda_k(\Xb)>t) \sim \frac{1}{N} \P(\lambda_1(\Xb)>t). 
\]
We have 
\[
\Xb_{11} = \sum_{k=1}^N \lambda_k(\Xb) \Ub_{1k}^2 = \Ub_{11}^2\lambda_1(\Xb)+R,
\]
where $\Ub$ and $(\lambda_i(\Xb))_{i=1}^N$ are independent. 
Since $R\leq \lambda_2(\Xb)$, we have 
\begin{align}\label{eq:smallR}
\P(R>t) =o\big(\P(\lambda_1(\Xb)>t)\big)=o\big(\P(\Xb_{11}>t)\big).
\end{align}
Next, since $R\ge 0$,
\begin{equation}\label{eq:lower}
\P(\Ub_{11}^2\lambda_1(\Xb)>t)\le \P(\Ub_{11}^2\lambda_1(\Xb)+R>t)=\P(\Xb_{11}>t).
\end{equation}
We have for any $\varepsilon\in(0,1)$,
\begin{equation}\label{eq:upper_split}
\P(\Ub_{11}^2\lambda_1(\Xb)>t) \geq \P(\Xb_{11}>t/(1-\eps))-\P(R>\varepsilon t/(1-\eps)).
\end{equation}
Divide \eqref{eq:upper_split} by $\P(\Xb_{11}>t)$ and use \eqref{eq:smallR} and the regular variation of the tail of $\Xb_{11}$, to obtain
\[
 \liminf_{t\to\infty}\frac{\P(\Ub_{11}^2\lambda_1(\Xb)>t)}{\P(\Xb_{11}>t)}\geq (1-\eps)^\eta. 
\]
Letting $\eps\to0^+$ and combining with \eqref{eq:lower} gives
\[
\P(\Xb_{11}>t) = \P( \lambda_1(\Xb) \Ub_{11}^2 + R>t) \sim \P( \Ub_{11}^2\lambda_1(\Xb) >t). 
\]

Fix $\alpha>0$. Let $Y$ be a positive random variable with $\E[Y^{\alpha+\eps}]<\infty$ for some $\eps>0$. By the Converse Breiman Lemma, 
\cite[Theorem 4.2]{ConvBreiman}, under the assumption $\E[Y^{\alpha+i \theta}]\neq 0$ for all $\theta\in \R$, one has the implication 
\[
t\mapsto \P(Y X>t) \in \mathcal{R}_{-\alpha} \implies  t\mapsto \P(X>t) \in \mathcal{R}_{-\alpha},
\]
where $\mathcal{R}_\rho$ denotes  the class of regularly
varying functions with index $\rho$. In such case 
\[
\P(Y X>t) \sim \E[Y^\alpha] \P(X>t). 
\]
We apply this result with $\alpha=\eta$, $X=\lambda_1(\Xb)$ and $Y=\Ub_{11}^2\sim\mathrm{Beta}(\frac{1}{2},\frac{N-1}{2})$.  We have for all $\eta>-1/2$ and $\theta\in\R$,
\[
\E[\Ub_{11}^{2\eta+2\theta i}] = \frac{\Gamma\left(N/2\right)\, \Gamma\left(\eta+i\theta+\tfrac12\right)}{\Gamma\left(\frac12\right)\, \Gamma\left(\eta+i\theta +N/2\right)},
\]
which is non-zero, since $\Gamma$ function has no zeros and the only singularities are the poles at $0, -1, -2,\ldots$ (here $\eta+i\theta +N/2$ is never a pole). Thus, we obtain 
\[
\P(\lambda_1(\Xb)>t) \sim \frac{1}{\E[\Ub_{11}^{2\eta}]} \P(\Ub_{11}^2 \lambda_1(\Xb)>t) \sim \frac{1}{\E[\Ub_{11}^{2\eta}]} \P(\Xb_{11}>t).
\]
\end{proof}

\begin{example}
Assume that  $\Ab\sim \mathcal{B}'_{\alpha,\alpha+\beta}(N)$. Then, $\Ab$ is orthogonally invariant and we have 
\[
h_k(s)
:= \E\big[\det(\Ab^{[k]})^{s/k}\big].
\]
Among the functions $h_k$, only $h_1$ and $h_2$ appear in condition (iii-a) and (iii-d) of
Theorem \ref{thm:tailsmat}.

By Lemma \ref{lem:prime_subm},  $\Ab\sim \mathcal{B}'_{\alpha,\alpha+\beta}(N)$ with
$\alpha,\beta>d_N-1=(N-1)/2$, then
$\Ab^{[k]}\sim \mathcal{B}'_{\alpha,\alpha+\beta-(N-k)/2}(k)$. 
Then,
\begin{align*}
 \\
h_k(s)&= \int_{\mathrm{Sym}_{>0}(k)}
\frac{\det(\xb)^{s/k}}{\mathrm{B}_k(\alpha,\alpha+\beta-(N-k)/2)}
\frac{\det(\xb)^{\alpha-d_k}}{\det(\id_k+\xb)^{2\alpha+\beta-(N-k)/2}}\,\dd\xb \\
&= \frac{\mathrm{B}_k\!\left(\alpha+s/k,\ \alpha+\beta-(N-k)/2-s/k\right)}
{\mathrm{B}_k\!\left(\alpha,\ \alpha+\beta-(N-k)/2\right)}.
\end{align*}
In particular, $h_k(s)=1$ if and only if
\[
s\in\left\{0,\ k\Bigl(\beta-\frac{N-k}{2}\Bigr)\right\}.
\]
Set
\[
\eta_k := k\Bigl(\beta-\frac{N-k}{2}\Bigr).
\]

Since $\beta>(N-1)/2$, we have $\eta_1<\eta_2$. Hence condition (iii-d) holds by convexity of $h_2$, since
\[
h_2(\eta_1)=\E\left[\det(\Ab^{[2]})^{\eta_1/2}\right] < 1 = h_2(\eta_2).
\]
Therefore, by Theorem \ref{thm:tailsmat},
\[
\mu_\Xb\big((t,+\infty)\big) \sim c\, t^{-\eta_1},\qquad t\to+\infty,
\]
where the constant $c$ is explicit.

The remaining indices $\eta_k$ naturally arise as tail indices for other related quantities. For instance, one may show by direct calculation that
\[
\P\big(\lambda_i(\Xb^{[k]})>t\big)\sim c_{i,k}\, t^{-\eta_i}.
\]
Moreover, since a solution $\Xb$ to \eqref{eq:BetaPerp} is $\mathcal{B}'_{\alpha,\beta}(N)$-distributed, we have
\[
\Xb^{[k]}\sim \mathcal{B}'_{\alpha,\beta-(N-k)/2}(k).
\]
Let $C_k\subset \mathrm{Sym}_{>0}(k)$ be a Borel set bounded away from $0$, and let
$\nu_k(\dd\xb)=\det(\xb)^{-d_k}\dd\xb$, $d_k = (k+1)/2$, denote the $\mathrm{GL}(k,\R)$-invariant measure on $\mathrm{Sym}_{>0}(k)$.
Then, as $t\to\infty$,
\begin{align*}
\P(\Xb^{[k]}\in t\,C_k)
&=\frac{1}{\mathrm{B}_k(\alpha,\beta-\frac{N-k}{2})}
\int_{t\,C_k}\frac{\det(\xb)^\alpha}{\det(\id_k+\xb)^{\alpha+\beta-\frac{N-k}{2}}}\,\nu_k(\dd\xb) \\
&=\frac{1}{\mathrm{B}_k(\alpha,\beta-\frac{N-k}{2})}
\int_{C_k}\frac{\det(t\yb)^\alpha}{\det(\id_k+t\yb)^{\alpha+\beta-\frac{N-k}{2}}}\,\nu_k(\dd\yb) \\
&= \frac{t^{k \,\alpha}}{t^{k(\alpha+\beta-\frac{N-k}{2})}}\,
\frac{1}{\mathrm{B}_k(\alpha,\beta-\frac{N-k}{2})}
\int_{C_k}\frac{\det(\yb)^\alpha}{\det(t^{-1}\id_k+\yb)^{\alpha+\beta-\frac{N-k}{2}}}\,\nu_k(\dd\yb) \\
&\sim t^{-\eta_k}\,
\frac{1}{\mathrm{B}_k(\alpha,\beta-\frac{N-k}{2})}
\int_{C_k}\det(\yb)^{-\beta+\frac{N-k}{2}}\,\nu_k(\dd\yb),
\end{align*}
where the last step follows from Lebesgue’s dominated convergence theorem.

Since our primary goal is to study the expected empirical spectral distribution $\mu_\Xb$, we do not pursue these related tail asymptotics further here. 
\end{example}

\section{Weak convergence}\label{sec:weak}

In this subsection, we study weak convergence of the expected empirical eigenvalue distributions of matrix perpetuities with the size $N\to\infty$ to the corresponding distribution of a free perpetuity studied in \cite{FreePerp}.

In order to motivate developments in this section, we make further observations in the framework of the Example \ref{ex:1}. 
\begin{example}
    We are interested in the properties of the solution from Example \ref{ex:1} as $N\to+\infty$. We choose $\alpha=\alpha_N$ and $\beta=\beta_N$ so that the limiting objects are non-trivial. Let us first describe the weak convergence of the empirical spectral distribution of $\mathcal{B}'_{\alpha_N, \beta_N}(N)$. We discuss the convergence of empirical spectral distribution of Beta matrices in the appendix.

Assume that 
\begin{align}\label{eq:alphabetaN}
\lim_{N\to+\infty}\frac{\alpha_N}{N} =  \frac{a}{2}\quad\mbox{and}\quad  \lim_{N\to+\infty}\frac{\beta_N}{N} =  \frac{b}{2},
\end{align}
where  $(a,b)\in[1,+\infty)\times[1,+\infty)$ with $ab>1$. 
The empirical spectral distribution of $\mathcal{B}'_{\alpha_N, \beta_N}(N)$ converges weakly almost surely to the probability measure $f\mathcal{B}'_{a,b}(\dd x)=f_{a,b}(x)\dd x$ defined by
\begin{align}\label{eq:fab}
f_{a,b}(x)=\frac{(b-1)\sqrt{(\gamma_+-x)(x-\gamma_-)}}{2\pi x(1+x)}I_{[\gamma_-,\gamma_+]}(x)
\end{align}
with
\[
\gamma_\pm = \left(\frac{\sqrt{ab}\pm\sqrt{a+b-1}}{b-1}\right)^2.
\]
The same convergence holds true for the expected empirical spectral distribution.

   Assume that for each $N\in\mathbb{N}$, $\Ab^{(N)}\sim \mathcal{B}'_{\alpha_N, \alpha_N+\beta_N}(N)$, where $(\alpha_N,\beta_N)$ are chosen according to \eqref{eq:alphabetaN}. 
   
   Let 
    \begin{align*}
\Xb^{(N)} \stackrel{d}{=} (\Ab^{(N)})^{1/2} \Xb^{(N)}(\Ab^{(N)})^{1/2}+\Ab^{(N)},\qquad \Ab^{(N)}\mbox{ and }\Xb^{(N)} \mbox{ are independent}.
\end{align*}
From Example \ref{ex:1} we know that $\Xb^{(N)}\sim\mathcal{B}'_{\alpha_N, \beta_N}(N)$. Thus, the expected empirical spectral distributions satisfy
\[
\mu_{\Ab^{(N)}}\stackrel{w}{\longrightarrow} f\mathcal{B}'_{a,a+b}\quad\mbox{and}\quad\mu_{\Xb^{(N)}}\stackrel{w}{\longrightarrow} f\mathcal{B}'_{a,b}.
\]

In \cite{FreePerp} we studied free perpetuities, and the related free perpetuity is defined as follows.
 Assume that $\A\sim f\mathcal{B}'_{a,a+b}$. Then, $\X\sim f\mathcal{B}'_{a,b}$ is the unique solution to
\begin{align}\label{eq:new}
	\X \stackrel{d}{=} \A^{1/2}\X\A^{1/2}+\A,\qquad \A\mbox{ and }\X\mbox{ are freely independent}.
\end{align}
Thus we obtain that empirical spectral distributions of matrix perpetuities converge weakly to the distribution of the corresponding free perpetuity.

\end{example}

\begin{thm}\label{thm:matmodel}
Fix $\A,\B$ in a $W^*$ probability space $(\mathcal{A},\tau)$. Assume that for each $N\in\mathbb{N}$,
\begin{enumerate}
\item[(i)] $\Ab^{(N)}, \Bb^{(N)}\in\mathrm{Sym}_{\geq0}(N)$,
\item[(ii)] $(\ub_N\Ab^{(N)} \ub_N^\top, \ub_N \Bb^{(N)} \ub_N^\top)\stackrel{d}{=}(\Ab^{(N)},\Bb^{(N)})$ for any orthogonal matrix $\ub_N$, 
\item[(iii)] Assume that for any fixed $L\geq 1$ families 
\[
\left\{\left(\left(\Ab^{(N)}_1\right)^{1/2},\Bb^{(N)}_1\right),\ldots,\left(\left(\Ab^{(N)}_L\right)^{1/2},\Bb^{(N)}_L\right)\right\}
\]
of independent copies of $(\Ab^{(N)},\Bb^{(N)})$ are asymptotically free, i.e., converge in noncommutative distribution to the family
\[\left\{(\A_1^{1/2},\B_1),\ldots,(\A_L^{1/2},\B_L)\right\}\]
of free copies of $(\A^{1/2},\B)$.
\item[(iv)]  $\tau(\A)<1$. 
\end{enumerate}
Then, for sufficiently large $N$, there exists a unique solution to 
    \begin{align}\label{eq:freeperp_}
\Xb^{(N)}\stackrel{d}{=}(\Ab^{(N)})^{1/2} \Xb^{(N)} (\Ab^{(N)})^{1/2} +\Bb^{(N)},\quad (\Ab^{(N)},\Bb^{(N)})\mbox{ and }\Xb^{(N)}\mbox{ are independent},
\end{align}
Moreover,  the expected spectral distribution of $\Xb^{(N)}$ converges weakly to the distribution of $\X$, which is a unique solution to
    \begin{align}\label{eq:freeperp2}
     \X\stackrel{d}{=}\A^{1/2} \X \A^{1/2}+\B,\qquad (\A,\B)\mbox{ and }\X\mbox{ are *-free.}    
    \end{align}
    
\end{thm}

\begin{proof}
Observe that 
\begin{align*}
    \E[\log \Ab_{11}^{(N)}]\leq \log \E [ \Ab_{11}^{(N)}]
\end{align*}
and since $\E[\Ab_{11}^{(N)}]\to\tau(\A)<1$, for $N$ big enough $\gamma_N=\E[\log \Ab_{11}^{(N)}]<0$. Similarly convergence of $\Bb^{(N)}$ guarantees that $\E[\log^+(\Bb^2)_{11}]<\infty$. By Corollary \ref{cor:matrixperp_inv} for big enough $N$ matrix perpetuity $\Xb^{(N)}$ exists and is given by the series.

Set
\[
\Pi_n^{(N)} = (\Ab^{(N)}_1)^{1/2}\ldots (\Ab^{(N)}_n)^{1/2},\quad \Pi_0^{(N)} = \id_N. 
\]
and
\[
\Pi_n = (\A_1)^{1/2}\ldots (\A_n)^{1/2},\quad \Pi_0 = \bf 1. 
\]
\smallskip
\noindent
We consider variables given by truncated sums, which define unique solution to perpetuity equations both on the level of matrices and free variables. For $L\ge1$ define
\[
\Xb^{(N)}_L
:=
\sum_{k=1}^{L} \Pi_{k-1}^{(N)}\Bb_k^{(N)} (\Pi_{k-1}^{(N)})^\top,\qquad \Xb^{(N)}
:=
\sum_{k=1}^{\infty} \Pi_{k-1}^{(N)}\Bb_k^{(N)} (\Pi_{k-1}^{(N)})^\top.
\]
Similarly for $L\ge1$ define
\[
\X_L:=\sum_{k=1}^{L}\Pi_{k-1}\B_k \Pi_{k-1}^\ast,\qquad \X:=\sum_{k=1}^{\infty}\Pi_{k-1}\B_k \Pi_{k-1}^\ast.
\]
Since each summand is positive, both $\Xb^{(N)}_L$ and $\X_L$ are positive.

Let
\[
\alpha_N = \frac{1}{N}\E[\mathrm{Tr}(\Ab^{(N)})]  \quad \mbox{and}\quad \beta_N:=\frac1N\,\E[\mathrm{Tr}(\Bb^{(N)})].
\]
By $(iii)$, we have
\[
\alpha_N\to \tau(\A)\quad\mbox{and}\quad 
\beta_N\to \tau(\B),
\]
hence
\[
\beta_\ast:=\sup_{N\ge1} \beta_N<\infty
\]
and since $\tau(\A)<1$, there exists $N_0$ and $\alpha_\ast<1$ such that for $N\geq N_0$,
\[
\alpha_N \leq \alpha_\ast. 
\]

We clearly have
\[
\frac1N\mathrm{Tr}(\Xb^{(N)}-\Xb^{(N)}_L)
=
\sum_{k=L+1}^{\infty}
\frac1N\mathrm{Tr}\!\left(
{\bf \Pi}_{k-1}^{(N)}\Bb_k^{(N)}\big({\bf \Pi}_{k-1}^{(N)}\big)^\top
\right).
\]
Orthogonal invariance implies that
\[
\E[\Ab^{(N)}] = \alpha_N \id_N\quad\mbox{and}\quad \E[\Bb^{(N)}] = \beta_N \id_N.
\]
Hence, by independence, conditioning first on $\Pi_{k-1}^{(N)}$, next on $\Pi_{k-2}^{(N)}$, up to $\Pi_{1}^{(N)}$, for sufficiently large $N$ we have
\[
\E\!\left[
\frac1N\mathrm{Tr}\!\left(
\Pi_{k-1}^{(N)}\Bb_k^{(N)}\big(\Pi_{k-1}^{(N)}\big)^\top
\right)\right] = \alpha_N^{k-1}\beta_N
\le  \alpha_\ast^{k-1} \beta_\ast.
\]
Therefore
\begin{equation}\label{eq:tail_trace_bound_N}
\E\!\left[\frac1N\mathrm{Tr}(\Xb^{(N)}-\Xb^{(N)}_L)\right]
\le
\beta_\ast \sum_{k=L+1}^{\infty}\alpha_\ast^{k-1}
=
\frac{\beta_\ast}{1-\alpha_\ast} \alpha_\ast^L.
\end{equation}
In particular, for sufficiently large $M>0$ we have
\begin{equation}\label{eq:first_moment_truncN}
\sup_{N\ge M}\E\!\left[\frac1N\mathrm{Tr}(\Xb^{(N)}_L)\right]
\le
\beta_\ast \sum_{k=1}^{L}\alpha_\ast^{k-1}
<\infty.
\end{equation}

On the free side, we have
\[
\tau(\Pi_{k-1}\B_k \Pi_{k-1}^\ast )=\tau(\A)^{k-1} \tau(\B),\qquad k\ge0.
\]
Hence
\[
\sum_{k=1}^{\infty}\tau(\Pi_{k-1}\B_k \Pi_{k-1}^\ast )
=
\tau(\B)\sum_{k=0}^{\infty}\tau(\A)^k
<
\infty,
\]
because $\tau(\A)<1$. Therefore the positive series
\[
\sum_{k=1}^\infty \Pi_{k-1}\B_k \Pi_{k-1}^\ast 
\]
converges in $L^1(\tau)$ to a positive element, which we denote by $\X$.
Moreover, as in the matrix case, we have
\begin{equation}\label{eq:tail_trace_bound_free}
\tau(\X-\X_L)
=
\sum_{k=L+1}^{\infty}\tau(\Pi_{k-1}\B_k \Pi_{k-1}^\ast )
=
\frac{\tau(\B)}{1-\tau(\A)}\,\tau(\A)^L.
\end{equation}
By construction $\X$ satisfies \eqref{eq:freeperp2}, and by
Theorem~\ref{thm:uniqe_free} its distribution is uniquely determined.

For every $L\ge1$, let
\[
\mu_N= \mu_{\Xb^{(N)}}\quad\mbox{and}\quad
\mu_{N,L}
:= \mu_{\Xb_L^{(N)}},
\]
be the expected spectral distributions of $\Xb^{(N)}$ and $\Xb_L^{(N)}$ respectively.
Of course $\Xb^{(N)}_L$ is a noncommutative polynomial in $({\Ab}_1^{(N)})^{1/2},\Bb_1^{(N)},\ldots,({\Ab}_L^{(N)})^{1/2},\Bb_L^{(N)}$.

Observe that assumption $(iii)$ implies that  moments of $\mu_{\Xb_L^{(N)}}$ converge to the respective moments of $\mu_{\X_L}$, and since $\mu_{\X_L}$ is compactly supported by definition (as a polynomial in bounded operators) we get
\begin{equation}\label{eq:weak_trunc}
\mu_{N,L}\stackrel{w}{\longrightarrow} \mu_{\X_L} \quad \mbox{as } N\to\infty.
\end{equation}

For $z\in\mathbb C_+$ define
\[
m_N(z):=\int\frac{1}{x-z}\,\mu_N(\dd x)
=
\E\!\left[\frac1N\mathrm{Tr}(\Xb^{(N)}-z\id_N)^{-1}\right],
\]
\[
m_{N,L}(z):=\int\frac{1}{x-z}\,\mu_{N,L}(\dd x)
=
\E\!\left[\frac1N\mathrm{Tr}(\Xb^{(N)}_L-z\id_N)^{-1}\right],
\]
and
\[
m(z):=\tau((\X-z)^{-1}),
\qquad
m_L(z):=\tau((\X_L-z)^{-1}).
\]

By \eqref{eq:weak_trunc},  we have  for fixed $L$,
\begin{equation}\label{eq:st_trunc}
m_{N,L}(z)\longrightarrow m_L(z)
\qquad \mbox{as }N\to\infty.
\end{equation}

For any matrices ${\bf a}, {\bf b}$, we have 
\[
({\bf a}-z\id_N)^{-1}-({\bf b}-z\id_N)^{-1}
=
({\bf a}-z\id_N)^{-1}({\bf b}-{\bf a})({\bf b}-z\id_N)^{-1}.
\]
using (see for example chapter A.3 in \cite{AGZBook}) $|\mathrm{Tr}({\bf a} )|\leq \lVert {\bf a}\rVert_1$ and $\lVert {\bf a b}\rVert_1\leq \lVert {\bf a}\rVert_\infty \rVert {\bf b}\rVert_1$ holding for any two matrices ${\bf a}, {\bf b}$ we have
\[
\left|
\frac1N\mathrm{Tr}({\bf a}-z\id_N)^{-1}-\frac1N\mathrm{Tr}({\bf b}-z\id_N)^{-1}
\right|
\le
\frac{1}{(\Im z)^2}\,\frac1N\|{\bf a}-{\bf b}\|_1.
\]
Applying this with $({\bf a},{\bf b})=(\Xb^{(N)},\Xb^{(N)}_L)$, taking expectation and using that
$\Xb^{(N)}-\Xb^{(N)}_L\ge0$, we obtain
\begin{equation}\label{eq:resolvent_tail_N_lyap}
|m_N(z)-m_{N,L}(z)|
\le
\frac{1}{(\Im z)^2}\,
\E\!\left[\frac1N\mathrm{Tr}(\Xb^{(N)}-\Xb^{(N)}_L)\right]
\le
\frac{\beta_\ast}{(1-\alpha_\ast)(\Im z)^2}\,\alpha_\ast^L
\end{equation}
by \eqref{eq:tail_trace_bound_N}.

Likewise, since $\X-\X_L\ge0$,
\begin{equation}\label{eq:resolvent_tail_free_lyap}
|m(z)-m_L(z)|
\le
\frac{1}{(\Im z)^2}\,\tau(\X-\X_L)
=
\frac{\tau(\B)}{(1-\tau(\A))(\Im z)^2}\,\tau(\A)^L
\end{equation}
by \eqref{eq:tail_trace_bound_free}.

For every fixed $z\in\mathbb C_+$ we have
\[
|m_N(z)-m(z)|
\le
|m_N(z)-m_{N,L}(z)|
+
|m_{N,L}(z)-m_L(z)|
+
|m_L(z)-m(z)|.
\]
Hence using \eqref{eq:st_trunc}, \eqref{eq:resolvent_tail_N_lyap}, and
\eqref{eq:resolvent_tail_free_lyap}, we get 
\[
\limsup_{N\to\infty}|m_N(z)-m(z)|
\le
\frac{\beta_\ast}{(1-\alpha_\ast)(\Im z)^2}\,\alpha_\ast^L
+
\frac{\tau(\B)}{(1-\tau(\A))(\Im z)^2}\,\tau(\A)^L.
\]
Letting $L\to\infty$ yields
\[
m_N(z)\longrightarrow m(z)
\qquad\text{for every }z\in\mathbb C_+.
\]
The weak convergence of the expected empirical spectral
distribution of $\Xb^{(N)}$ follows.
\end{proof}

In \cite{FreePerp}, we studied tail asymptotics of free perpetuities. In particular, we showed that in the critical case \(\tau(\A)=1\), with \(\tau(\A^2)<\infty\), and under additional technical assumptions, the perpetuity \(\X\) exists and has universal tail asymptotics:
\begin{align}\label{eq:tail1}
    \lim_{t\to+\infty} t^{1/2}\mu_{\X}\big((t,+\infty)\big)
    =
    \frac{2\sqrt{2}}{\pi}
    \sqrt{\frac{\tau(\B)}{\mathrm{Var}(\A)}}.
\end{align}
Theorem~\ref{thm:matmodel}, however, excludes this critical regime by assumption~(iv). It would be interesting to extend the theorem so as to cover this case as well. We emphasize, however, that there is no a priori reason to expect the tail exponents of finite-dimensional matrix perpetuities to converge to the tail exponent of the limiting expected empirical spectral distribution.

\appendix

\section{Weak convergence of the spectral distribution of matrix Beta prime distribution}\label{app:A}

Let $\mathrm{Herm}^1(N)=\mathrm{Sym}(N, \mathbb{R})$ and $\mathrm{Herm}^2(N)=\mathrm{Herm}(N, \mathbb{C})$ denote the space of $N\times N$ symmetric real and complex Hermitian matrices, respectively.

Let $d_N=N^{-1}\dim\mathrm{Herm}^d(N)=1+d (N-1)/2$.
The matrix beta-prime distribution with parameters $\alpha,\beta>d_N-1$ on the cone of positive definite matrices $\mathrm{Herm}^d_+(N) = \{ M\in \mathrm{Herm}^d(N)\colon M>0\}$, $d=1,2$, is defined by its density
\[
\mathcal{B}'_{\alpha,\beta}(\mathrm{Herm}^d_+(N) )(\dd\xb) \propto \frac{\det(\xb)^{\alpha-d_N}}{\det(\id_N+\mathbf{x})^{\alpha+\beta}}I_{\mathrm{Herm}_+^d(N)}(\xb)\dd\xb
\]
where $\dd\mathbf{x}$ is the Euclidean measure on $\mathrm{Herm}^d(N)$. 

Consider $(\alpha,\beta)=(\alpha_N,\beta_N)$ satisfying 
\begin{align}\label{eq:alphabetaN2}
\lim_{N\to+\infty}\frac{\alpha_N}{N} =  \frac{a\, d}{2}\quad\mbox{and}\quad  \lim_{N\to+\infty}\frac{\beta_N}{N} =  \frac{b\, d}{2}.
\end{align}
Since the distribution of $\mathcal{B}'$ is invariant under the action of orthogonal ($d=1$) or unitary ($d=2$) groups, the eigenvalue distribution of $\mathcal{B}'_{\alpha_N, \beta_N}(\mathrm{Herm}_+^d(N))$, $a\geq1$ and $b>1$,  has density proportional to
\begin{align}\label{eq:LambdaDens}
 \exp\left( - \frac{d\,N}{2}
\sum_{k=1}^N V_N(\lambda_k)\right) \prod_{1\leq i<j\leq N} (\lambda_i-\lambda_j)^{d} I_{\lambda_1>\ldots>\lambda_N},
\end{align}
where 
\[
V_N(x)=\frac{\alpha_N+\beta_N}{d N/2}\log(1+x) -  \left(\frac{\alpha_N}{d N/2}- \frac{d_N}{N d/2}\right)\log(x)
\]
for $x>0$ and  $V_N(x)=+\infty$ for $x\leq 0$. As $N\to+\infty$ we have $V_N(x)\to V(x)$, $x>0$, where
\begin{align}\label{eq:defV}
V(x)=(a+b)\log(1+x) -  (a-1)\log(x).
\end{align}
\begin{lemma}\label{lem:minimizer}
Let $a\geq1$ and $b>1$ and assume that $(\lambda_1,\ldots,\lambda_N)$ has density given by \eqref{eq:LambdaDens}. Then, the (random) probability measure 
\[
\hat\mu_N = \frac{1}{N}\sum_{k=1}^N \delta_{\lambda_k}
\]
converges weakly a.s. to the unique non-random minimizer of 
$E_V\colon \mathcal{M}_c\to \R\cup\{+\infty\}$, where
\begin{align}\label{eq:EV}
E_V(\mu)= \int_\R \int_\R \log\left(\tfrac{1}{|x-y|}\right)\mu(\dd x)\mu(\dd y)+\int_\R V(x)\mu(\dd x),
\end{align}
where $V$ is defined in \eqref{eq:defV}
and $\mathcal{M}_c$ is the set of probability measures on $\R$ with compact support. 
\end{lemma}
\begin{remark}
    Note that the pushforward measure of $\mathcal{B}'_{\alpha,\beta}(\mathrm{Herm}^d_+(N) )$ by the mapping $\xb\mapsto \xb^{-1}$ is $\mathcal{B}'_{\beta,\alpha}(\mathrm{Herm}^d_+(N) )$. Therefore, Lemma \ref{lem:minimizer} also holds when $a>1$ and $b\geq 1$. 
\end{remark}
\begin{proof}[Proof of Lemma \ref{lem:minimizer}]
Since $(V_N)_{N\geq 1}$ is not a constant sequence, we will use the results of \cite{Feral}, more precisely, we check the assumptions of \cite[Theorem 2.1]{Feral}.

    We observe that $V_N\colon (0,+\infty)\to\R$ is continuous for all $N\in\mathbb{N}$ and that convergence $V_N\to V$ is uniform on compact subsets of $(0,+\infty)$. By \eqref{eq:alphabetaN2}, $\beta_N/(d N)>1/2$ for large $N$. Thus, for large $N$, we have 
    \[
    \lim_{x\to+\infty} \frac{V_N(x)}{\log(1+x^2)} = \frac{\beta_N+d_N}{d N} > \frac{1}{2}+\frac{d_N}{d N} = 1+\frac{1-d/2}{d N} \geq 1.
    \]
 Thus, there exists $\eps>0$ such that for sufficiently large $N$ and $T$ one has
\[
V_N(x)\geq(1+\eps)\log(1+x^2),\qquad x\geq T.
\]
If $(V_N)_{N\geq 1}$ satisfied these conditions on $\R$, then the result would follow directly from \cite[Theorem 2.1]{Feral}. The second paragraph after the formulation of \cite[Theorem 2.1]{Feral} states that this result remains true if one restricts the potential to the positive half-line as in our case.
\end{proof}
\begin{remark}
If $(\lambda_k)_{k=1}^N$ are the eigenvalues of a matrix $\Yb^{(N)}$, then expected measure of $\hat\mu_N$ defined in Lemma \ref{lem:minimizer}, that is, $\mu_{\Yb^{(N)}}(\cdot)=\E[\mu_N(\cdot)]$ coincides with the expected empirical spectral distribution of $\Yb^{(N)}$.  Through Lebesgue's dominated convergence theorem, almost sure weak convergence of $\mu_N$ implies weak convergence of $\mu_{\Yb^{(N)}}$ to the same limit. 
\end{remark}

\begin{lemma} \label{lem:minimizer2}
If $a\geq1$ and $b>1$, then $E_V$ defined in \eqref{eq:EV} is minimized by $\mu_{a,b}^\ast(\dd x)=f_{a,b}(x)\dd x$, where $f_{a,b}$ is given in \eqref{eq:fab}.
\end{lemma}
\begin{proof}
The measure $\mu_{a,b}^\ast$ coincides with the free Beta prime distribution defined in \cite{Yoshida}. By \cite[Proposition 3.2]{Yoshida}, its  Cauchy transform for $z\in\mathbb{C}^+ = \{z\in \mathbb{C}\colon \mathrm{Im}\,z>0\}$ equals
\begin{align*}
G(z) &= 
\frac{(b+1)z+(1-a)-\sqrt{(b-1)^2 z^2-2(ab+a+b-1)z+(a-1)^2}}{2 z(1+z)},\\
&=\frac{(b+1)z+(1-a)-(b-1)\sqrt{(z-\gamma_-)(z-\gamma_+)}}{2z(1+z)},
\end{align*}
where the branch of the square root is chosen so that $\lim_{z\to+\infty} G(z)=0$. 
It is a standard result that from the Cauchy transform one can obtain the Hilbert transform of $f_{a,b}$, see e.g. \cite[Chapter 3.1]{HP00}. We get (we drop the usual $\pi^{-1}$ factor in the definition of the Hilbert transform $\mathcal{H}$)
\begin{align*}
\mathcal{H}(f_{a,b})(y)&=\mathrm{P.V.}\int_\R \frac{f_{a,b}(x)}{y-x}\dd x = \lim_{\eps\to 0+} \mathrm{Re} \,G(y+i \eps) \\
&= \begin{cases}
\frac{(b+1)y+(1-a)+(b-1)\sqrt{(\gamma_--y)(\gamma_+-y)}}{2y(1+y)}, & y\in(0,\gamma_-),\\
    \frac{1}{2}V'(y),& y\in [\gamma_-,\gamma_+], \\
\frac{(b+1)y+(1-a)-(b-1)\sqrt{(y-\gamma_-)(y-\gamma_+)}}{2y(1+y)}, & y>\gamma_+,
\end{cases}
\end{align*}
where the case $y\in[\gamma_-,\gamma_+]$ follows from \cite[Section 4.3 (i)]{Yoshida} and the case $y\in(0,+\infty)\setminus [\gamma_-,\gamma_+]$ follows from analytic continuation of $G$. Moreover, we have 
\[
\mathcal{H}(f_{a,b})(y)<\frac{(b+1)y+(1-a)}{2y(1+y)}=\frac12V'(y),\qquad y>\gamma_+ 
\]
and similarly $\mathcal{H}(f_{a,b})(y)>\frac12V'(y)$ for $y\in(0,\gamma_-)$. We note that $\gamma_->0$ if and only if $a>1$. 

Let $U_{\mu^\ast}(y) =  -\int_{\R} \log(|x-y|)\mu^\ast(\dd x)$.  By \cite[Theorem 3.1]{ST97}, 
the unique compactly supported minimizer $\mu^\ast$ of $E_{V}$ is characterized by the following two conditions: there exists $C\in\R$ such that
    \begin{align}\label{eq:UV}\begin{cases}
       U_{\mu^\ast}(x) = -\frac12 V(x)+C,& x\in \mathrm{supp}(\mu^\ast), \\
         U_{\mu^\ast}(x)\geq -\frac12V(x)+C,& x\notin \mathrm{supp}(\mu^\ast).
\end{cases}    \end{align}
We will show that these conditions are satisfied for $\mu^\ast=\mu^{\ast}_{a,b}$. Let $U=U_{\mu^\ast_{a,b}}$.
We have for all $x<y$, 
\begin{align}\label{eq:UU}
\begin{split}
U(y)-U(x) &= \int_{\R} \log\left| \frac{x-t}{y-t}\right|f_{a,b}(t)\dd t = \int_{\R} \mathcal{H}( I_{[x,y]})(t) f_{a,b}(t)\dd t \\&= -\int_{\R} I_{[x,y]}(t) \mathcal{H}(f_{a,b})(t)\dd t
= -\int_x^y \mathcal{H}(f_{a,b})(t)\dd t,\end{split}
\end{align}
where we have used the fact that Hilbert transform is skew-adjoint operator relative to the duality pairing between $L^p(\R)$ and $L^q(\R)$, $1/p+1/q=1$, $p>1$, and an explicit formula for the Hilbert transform of the indicator function. It can be easily verified that $f_{a,b}\in L^p(\R)$ with $a>1$ for any $p>1$ and $f_{1,b}\in L^p(\R)$ for $p\in(1,2)$. 

If $x<0$, then $V(x)=+\infty$ and therefore $U(x)\geq-\frac{1}{2}V(x)+C$ for any $C\in\R$. The same argument applies to the case $a>1$ and $x=0$. 
Moreover, by \eqref{eq:UU} and previously established inequalities for $\mathcal{H}(f_{a,b})(y)$, we obtain
\begin{align*}
    U(x) &\geq  
     -\frac12 V(x)+ \frac12 V(\gamma_-)+U(\gamma_-),\qquad x\in(0,\gamma_-),\\
    U(x) &= 
    -\frac12 V(x)+ \frac12 V(\gamma_-)+U(\gamma_-),\qquad x\in[\gamma_-,\gamma_+]
 \intertext{and}
 U(x) &\geq 
 -\frac12 V(x)+ \frac12 V(\gamma_+)+U(\gamma_+),\quad\qquad x>\gamma_+.
\end{align*}
Since $U(\gamma_+)-U(\gamma_-) = -\frac12 ( V(\gamma_+)-V(\gamma_-))$, we obtain \eqref{eq:UV} with 
\[
C=\frac12 V(\gamma_-)+U(\gamma_-) = \frac12 V(\gamma_+)+U(\gamma_+).
\]
\end{proof}

\section{Submatrices of Beta prime distribution}\label{app:B}

\begin{lemma}\label{lem:prime_subm}
Let $\Ab\sim \mathcal{B}'_{\alpha,\beta}(N)$ with $\alpha,\beta>d_N-1=(N-1)/2$. Then, for every $k=1,\dots,N$,
\[
\Ab^{[k]}\sim \mathcal{B}'_{\alpha,\beta-(N-k)/2}(k).
\]
\end{lemma}

\begin{proof}
Let $\Xb\sim \cW_{\alpha,I_N}(N)$ and $\Yb\sim \mathcal{W}_{\beta,\id_N}(N)$ be independent (recall definition of Wishart distribution in \eqref{eq:Wishart}). By the standard Wishart representation of the matrix beta-prime law,
\[
\Ab \stackrel{d}= \Xb^{1/2}\Yb^{-1}\Xb^{1/2}.
\]
Since $\Yb^{-1}$ is orthogonally invariant, Corollary~\ref{prop:39} gives
\begin{align}\label{eq:A[k]}
\Ab^{[k]}
\stackrel{d}{=}
\bigl(\Xb^{[k]}\bigr)^{1/2}\,(\Yb^{-1})^{[k]}\,\bigl(\Xb^{[k]}\bigr)^{1/2}.
\end{align}
Also, it is a standard fact that $\Xb^{[k]}\sim \cW_{\alpha,\id_k}(k)$.

It remains to identify the law of $(\Yb^{-1})^{[k]}$. Write
\[
\Yb=
\begin{pmatrix}
\Yb_{11} & \Yb_{12}\\
\Yb_{21} & \Yb_{22}
\end{pmatrix},
\qquad \Yb_{11}=\Yb^{[k]}.
\]
By the block inversion formula,
\[
(\Yb^{-1})^{[k]}=(\Yb_{1\cdot 2})^{-1},
\qquad
\Yb_{1\cdot 2}:=\Yb_{11}-\Yb_{12}\Yb_{22}^{-1}\Yb_{21}.
\]
Now apply \cite[Proposition~3.2]{MassamWesolowski06}
to $\Yb$, with block sizes $r=k$ and $s=N-k$, and $c=\id_N$. It yields (recall the definition of matrix GIG distribution in \eqref{eq:GIG})
\[
(\Yb_{22},\Yb_{1\cdot 2})\mid \Yb_{12}
\sim
 \mathcal{G}_{\cdots}(N-k) \times \cW_{\beta-\frac{N-k}{2},\id_k}(k),
\]
where the parameters of the matrix GIG distributions are irrelevant for our purposes, and therefore we omit their specification here.
Since the second factor does not depend on $\Yb_{12}$, we obtain
\[
((\Yb^{-1})^{[k]})^{-1}= \Yb_{1\cdot 2}\sim \cW_{\beta-\frac{N-k}{2},\id_k}(k).
\]
Therefore, by the same Wishart representation of the beta-prime law on $\mathrm{Sym}_{>0}(k)$, \eqref{eq:A[k]}  implies that
\[
\Ab^{[k]}\sim \mathcal{B}'_{\alpha,\beta-(N-k)/2}(k).
\]
\end{proof}

\section{Measurable polar decomposition}\label{sec:C}

Let $(X,\mathcal A)$ be a measurable space and let $Y$ be a topological space. A set-valued map
\[
\Psi \colon X \to 2^Y
\]
is called weakly measurable if for every open set $U\subset Y$,
\[
\{x\in X\colon \Psi(x)\cap U\neq\emptyset\}\in\mathcal A.
\]

\begin{remark}
If $Y$ is a metric space and $\Psi(x)$ is compact for every $x\in X$, then weak measurability is equivalent to the following condition: for every closed set $F\subset Y$,
\[
\{x\in X\colon \Psi(x)\cap F\neq\emptyset\}\in\mathcal A.
\]
Indeed, if $F$ is closed and $F^\varepsilon:=\{y\in Y\colon d(y,F)<\varepsilon\}$, then for compact $\Psi(x)$,
\[
\Psi(x)\cap F\neq\emptyset
\quad\Longleftrightarrow\quad
\Psi(x)\cap F^\varepsilon\neq\emptyset
\ \text{ for every }\ \varepsilon>0.
\]
Hence
\[
\{x\colon \Psi(x)\cap F\neq\emptyset\}
=
\bigcap_{n=1}^\infty
\{x\colon \Psi(x)\cap F^{1/n}\neq\emptyset\},
\]
so measurability for closed sets follows from the definition. Conversely, if $U\subset Y$ is open, then
\[
U=\bigcup_{n=1}^\infty F_n,
\qquad
F_n:=\{y\in Y\colon d(y,Y\setminus U)\ge 1/n\},
\]
with each $F_n$ closed, and therefore
\[
\{x\colon \Psi(x)\cap U\neq\emptyset\}
=
\bigcup_{n=1}^\infty
\{x\colon \Psi(x)\cap F_n\neq\emptyset\}.
\]
\end{remark}

\begin{thm}[Kuratowski--Ryll-Nardzewski measurable selection theorem]
Let $X$ be a standard Borel space, let $Y$ be a Polish space, and let
\[
\Psi\colon X\to 2^Y
\]
be a set-valued map with nonempty closed values. If $\Psi$ is weakly measurable, then there exists a Borel measurable map
\[
f\colon X\to Y
\]
such that
\[
f(x)\in \Psi(x)\qquad\text{for all }x\in X.
\]
\end{thm}

\begin{lemma}\label{lem:measurable}
Let $\Ab$ be a real random matrix, and let
\[
|\Ab|=(\Ab\Ab^\top)^{1/2}.
\]
Then there exists a $\sigma(\Ab)$-measurable orthogonal matrix $\Ub$ such that
\[
\Ab=|\Ab|\,\Ub.
\]
\end{lemma}

\begin{proof}
Let $O(N)$ denote the group of $N\times N$ orthogonal matrices. Since $O(N)$ is a closed and bounded subset of $\mathrm{Mat}(N)$, it is a compact metric space.

For each ${\bf m}\in \mathrm{Mat}(N)$, define
\[
\Phi({\bf m}) := \{\vb\in O(N)\colon {\bf m}=|{\bf m}|\,\vb\}.
\]
By the polar decomposition for real square matrices, $\Phi({\bf m})$ is nonempty for every ${\bf m}\in\mathrm{Mat}(N)$.

We next show that each $\Phi({\bf m})$ is closed in $O(N)$. Fix ${\bf m}$, and let $(\vb_k)_{k\ge1}$ be a sequence in $\Phi({\bf m})$ such that $\vb_k\to\vb$ in $O(N)$. Since
\[
{\bf m}=|{\bf m}|\,\vb_k \qquad\text{for all }k,
\]
and matrix multiplication is continuous, passing to the limit yields
\[
{\bf m}=|{\bf m}|\,\vb.
\]
Hence $\vb\in\Phi({\bf m})$, so $\Phi({\bf m})$ is closed. Since $O(N)$ is compact, each $\Phi({\bf m})$ is in fact compact.

Now consider the graph of $\Phi$:
\[
\mathrm{Graph}(\Phi)
:=\{({\bf m},\vb)\in \mathrm{Mat}(N)\times O(N)\colon {\bf m}=|{\bf m}|\,\vb\}.
\]
The map
\[
({\bf m},\vb)\mapsto {\bf m}-|{\bf m}|\,\vb
\]
is continuous on $\mathrm{Mat}(N)\times O(N)$, because ${\bf m}\mapsto {\bf m}{\bf m}^\top$ is continuous, the map
\[
{\bf s}\mapsto {\bf s}^{1/2}
\]
is continuous on the cone of symmetric positive semidefinite matrices, and matrix multiplication is continuous. Therefore $\mathrm{Graph}(\Phi)$ is closed.

We claim that $\Phi$ is weakly measurable. Since $O(N)$ is a compact metric space, it is enough to show that for every closed set $F\subset O(N)$, the set
\[
E_F:=\{{\bf m}\in \mathrm{Mat}(N)\colon \Phi({\bf m})\cap F\neq\emptyset\}
\]
is Borel. Indeed, if $U\subset O(N)$ is open, then
\[
U=\bigcup_{n=1}^\infty F_n,
\qquad
F_n:=\bigl\{\vb\in O(N)\colon  \mathrm{dist}(\vb,O(N)\setminus U)\ge 1/n\bigr\},
\]
and each $F_n$ is closed. Hence
\[
\{{\bf m}\colon\Phi({\bf m})\cap U\neq\emptyset\}
=\bigcup_{n=1}^\infty E_{F_n},
\]
so measurability for open sets follows once each $E_F$ is Borel.

Now fix a closed set $F\subset O(N)$. Then
\[
E_F
=
\pi_1\bigl(\mathrm{Graph}(\Phi)\cap (\mathrm{Mat}(N)\times F)\bigr),
\]
where $\pi_1\colon \mathrm{Mat}(N)\times O(N)\to \mathrm{Mat}(N)$ is the projection onto the first coordinate. Since $\mathrm{Graph}(\Phi)$ is closed and $F$ is closed in the compact space $O(N)$, the set
\[
C_F:=\mathrm{Graph}(\Phi)\cap (\mathrm{Mat}(N)\times F)
\]
is closed in $\mathrm{Mat}(N)\times O(N)$.

We show that $\pi_1(C_F)$ is closed. Let $({\bf m}_k)$ be a sequence in $\pi_1(C_F)$ such that ${\bf m}_k\to{\bf m}$ in $\mathrm{Mat}(N)$. For each $k$ there exists $\vb_k\in F$ with $({\bf m}_k,\vb_k)\in C_F$. Since $F$ is compact, after passing to a subsequence we may assume that $\vb_k\to\vb\in F$. Then
\[
({\bf m}_k,\vb_k)\to ({\bf m},\vb)
\]
in $\mathrm{Mat}(N)\times O(N)$. Because $C_F$ is closed, we get $({\bf m},\vb)\in C_F$, and hence ${\bf m}\in \pi_1(C_F)$. Thus $\pi_1(C_F)=E_F$ is closed.

Therefore $\Phi$ is weakly measurable.

Finally, $\mathrm{Mat}(N)\cong\mathbb R^{N^2}$ is a standard Borel space and $O(N)$ is a compact metric space, hence Polish. By the Kuratowski--Ryll-Nardzewski measurable selection theorem, there exists a Borel measurable map
\[
f\colon \mathrm{Mat}(N)\to O(N)
\]
such that
\[
f({\bf m})\in \Phi({\bf m})
\qquad\text{for every }{\bf m}\in \mathrm{Mat}(N).
\]
Define
\[
\Ub:=f(\Ab).
\]
Then $\Ub$ is $\sigma(\Ab)$-measurable, $\Ub\in O(N)$ almost surely, and
\[
\Ab=|\Ab|\,\Ub
\]
almost surely. This completes the proof.
\end{proof}

\bibliographystyle{plain}

\bibliography{Bibl}

\begin{thebibliography}{10}

\bibitem{AGZBook}
Greg~W. Anderson, Alice Guionnet, and Ofer Zeitouni.
\newblock {\em An introduction to random matrices}, volume 118 of {\em
  Cambridge Studies in Advanced Mathematics}.
\newblock Cambridge University Press, Cambridge, 2010.

\bibitem{ABO23}
Jonas Arista, Elia Bisi, and Neil O'Connell.
\newblock Matrix {W}hittaker processes.
\newblock {\em Probab. Theory Related Fields}, 187(1-2):203--257, 2023.

\bibitem{ABO}
Jonas Arista, Elia Bisi, and Neil O'Connell.
\newblock Matsumoto-{Y}or and {D}ufresne type theorems for a random walk on
  positive definite matrices.
\newblock {\em Ann. Inst. Henri Poincar\'{e} Probab. Stat.}, 60(2):923--945,
  2024.

\bibitem{MatrixPolymer}
Guillaume Barraquand and Zikun Ouyang.
\newblock Stationary inverse-{W}ishart polymers.
\newblock {\em arXiv:2511.14375}, 2025.

\bibitem{FreePerp}
S.~Belinschi, B.~Ko{\l}odziejek, and K.~Szpojankowski.
\newblock Free {P}erpetuities {I}: {E}xistence, {S}ubordination and {T}ail
  {A}symptotics.
\newblock {\em arXiv:2503.10319}, 2025.

\bibitem{Remarkable}
Serban~T. Belinschi and Alexandru Nica.
\newblock On a remarkable semigroup of homomorphisms with respect to free
  multiplicative convolution.
\newblock {\em Indiana Univ. Math. J.}, 57(4):1679--1713, 2008.

\bibitem{BP92}
Philippe Bougerol and Nico Picard.
\newblock Strict stationarity of generalized autoregressive processes.
\newblock {\em Ann. Probab.}, 20(4):1714--1730, 1992.

\bibitem{BDM}
D.~Buraczewski, E.~Damek, and T.~Mikosch.
\newblock {\em Stochastic models with power-law tails}.
\newblock Springer Series in Operations Research and Financial Engineering.
  Springer, [Cham], 2016.
\newblock The equation $X=AX+B$.

\bibitem{Sim10}
Dariusz Buraczewski, Ewa Damek, and Yves Guivarc'h.
\newblock Convergence to stable laws for a class of multidimensional stochastic
  recursions.
\newblock {\em Probab. Theory Related Fields}, 148(3-4):333--402, 2010.

\bibitem{Sim09}
Dariusz Buraczewski, Ewa Damek, Yves Guivarc'h, Andrzej Hulanicki, and Roman
  Urban.
\newblock Tail-homogeneity of stationary measures for some multidimensional
  stochastic recursions.
\newblock {\em Probab. Theory Related Fields}, 145(3-4):385--420, 2009.

\bibitem{MY2024}
Reda Chhaibi and Manon Defosseux.
\newblock Matsumoto-{Y}or processes on {J}ordan algebras.
\newblock {\em Probab. Theory Relat. Fields}, 2025.

\bibitem{DM24}
E.~Damek and S.~Mentemeier.
\newblock Analysing heavy-tail properties of stochastic gradient descent by
  means of stochastic recurrence equations.
\newblock {\em J. Appl. Probab.}, page 1–25, 2026.

\bibitem{Feral}
D.~F\'{e}ral.
\newblock On large deviations for the spectral measure of discrete {C}oulomb
  gas.
\newblock In {\em S\'{e}minaire de probabilit\'{e}s {XLI}}, volume 1934 of {\em
  Lecture Notes in Math.}, pages 19--49. Springer, Berlin, 2008.

\bibitem{MATRIXKESTEN}
Tristan Gauti\'{e}, Jean-Philippe Bouchaud, and Pierre Le~Doussal.
\newblock Matrix {K}esten recursion, inverse-{W}ishart ensemble and fermions in
  a {M}orse potential.
\newblock {\em J. Phys. A}, 54(25):Paper No. 255201, 58, 2021.

\bibitem{Gol91}
Charles~M. Goldie.
\newblock Implicit renewal theory and tails of solutions of random equations.
\newblock {\em Ann. Appl. Probab.}, 1(1):126--166, 1991.

\bibitem{GLP16}
Y.~Guivarc'h and \'{E}. Le~Page.
\newblock Spectral gap properties for linear random walks and {P}areto's
  asymptotics for affine stochastic recursions.
\newblock {\em Ann. Inst. Henri Poincar\'{e} Probab. Stat.}, 52(2):503--574,
  2016.

\bibitem{HP00}
F.~Hiai and D.~Petz.
\newblock {\em The semicircle law, free random variables and entropy},
  volume~77 of {\em Mathematical Surveys and Monographs}.
\newblock American Mathematical Society, Providence, RI, 2000.

\bibitem{ML3}
Liam Hodgkinson, Zhichao Wang, and Michael~W. Mahoney.
\newblock Models of heavy-tailed mechanistic universality.
\newblock In {\em Forty-second International Conference on Machine Learning},
  2025.

\bibitem{MAnal}
Roger~A. Horn and Charles~R. Johnson.
\newblock {\em Matrix analysis}.
\newblock Cambridge University Press, Cambridge, second edition, 2013.

\bibitem{ConvBreiman}
Martin Jacobsen, Thomas Mikosch, Jan Rosi\'{n}ski, and Gennady Samorodnitsky.
\newblock Inverse problems for regular variation of linear filters, a
  cancellation property for {$\sigma$}-finite measures and identification of
  stable laws.
\newblock {\em Ann. Appl. Probab.}, 19(1):210--242, 2009.

\bibitem{Kes73}
H.~Kesten.
\newblock Random difference equations and renewal theory for products of random
  matrices.
\newblock {\em Acta Math.}, 131:207--248, 1973.

\bibitem{KMY}
B.~Ko{\l}odziejek.
\newblock The {M}atsumoto-{Y}or property and its converse on symmetric cones.
\newblock {\em J. Theoret. Probab.}, 30(2):624--638, 2017.

\bibitem{LW00}
G\'{e}rard Letac and Jacek Weso{\l}owski.
\newblock An independence property for the product of {GIG} and gamma laws.
\newblock {\em Ann. Probab.}, 28(3):1371--1383, 2000.

\bibitem{MassamWesolowski06}
H\'{e}l\`ene Massam and Jacek Weso{\l}owski.
\newblock The {M}atsumoto-{Y}or property and the structure of the {W}ishart
  distribution.
\newblock {\em J. Multivariate Anal.}, 97(1):103--123, 2006.

\bibitem{Newman86}
C.~M. Newman.
\newblock The distribution of {L}yapunov exponents: exact results for random
  matrices.
\newblock {\em Comm. Math. Phys.}, 103(1):121--126, 1986.

\bibitem{NicaSpeicherMult}
Alexandru Nica and Roland Speicher.
\newblock On the multiplication of free {$N$}-tuples of noncommutative random
  variables.
\newblock {\em Amer. J. Math.}, 118(4):799--837, 1996.

\bibitem{ST97}
E.~B. Saff and V.~Totik.
\newblock {\em Logarithmic potentials with external fields}, volume 316 of {\em
  Grundlehren der mathematischen Wissenschaften [Fundamental Principles of
  Mathematical Sciences]}.
\newblock Springer-Verlag, Berlin, 1997.
\newblock Appendix B by Thomas Bloom.

\bibitem{ShlyakhtenkoTao}
Dimitri Shlyakhtenko and Terence Tao.
\newblock Fractional free convolution powers.
\newblock {\em Indiana Univ. Math. J.}, 71(6):2551--2594, 2022.

\bibitem{Yoshida}
H.~Yoshida.
\newblock Remarks on a free analogue of the beta prime distribution.
\newblock {\em J. Theoret. Probab.}, 33(3):1363--1400, 2020.

\end{thebibliography}

\end{document}